\documentclass[preprint,11pt]{elsarticle}

\usepackage{natbib, changes}
\usepackage{url}
\usepackage{bm}
\usepackage{amsthm,amsmath,amsfonts,amssymb}
\usepackage{tikz}
\usepackage{graphicx}
\usepackage{enumitem}
\usepackage{verbatim}
\fontsize{10pt}{10pt}\selectfont
\fontsize{9pt}{9pt}\selectfont

\newtheorem{Corollary}{Corollary}
\newtheorem{Proposition}{Proposition}

\theoremstyle{remark}
\newtheorem{Example}{Example}
\newtheorem{Remark}{Remark}

\def\spin{{[-1,1]}}

\def\pgf{\emph{p.g.f.~}}
\def\pgfl{\emph{p.g.fl.~}}
\def\iid{\emph{i.i.d.}~}
\def\odd{{\text{\rm odd}}}
\def\even{{\text{\rm even}}}

\newpageafter{abstract}
\begin{document}
\begin{frontmatter}
\title{Gaussian Fields on a hypercube\\from Long Range Random Walks}
 \author{Robert Griffiths}

\affiliation{
organization={School of Mathematics},
            addressline={Monash University},
            city={Clayton},
            postcode={3800}, 
            state={Victoria},
            country={Australia}
         }
	
\ead{Bob.Griffiths@Monash.edu}
\date{}

\begin{abstract}
We consider a class of  Gaussian Free Fields  denoted by $(g_x)_{x \in {\cal V}_N}$,  where $ {\cal V}_N = \{0,1\}^N$ and $N\in \mathbb{Z}_+$. 
These fields are related to a general class of $N$-dimensional random walks on the hypercube, which are killed at a certain rate. The covariance structure of the Gaussian free field is determined by the Green function of these random walks. 
There exists a coupling such that the Gaussian free fields  ${\cal G}_N := \big (g_{x}\big )_{x \in {\cal V}_N}$  form a Markov chain where $N$ is time.
 If the $N$ entries of the random walk are exchangeable, then the random variables in the Gaussian field can be coupled with spin glass models.
A natural choice is to take the increments of the random walk to be from a de Finetti sequence with elements $\{0,1\}$. The random walk is then well defined on ${\cal V}_\infty$. The Green function and a strong representation for $(g_x)$ are characterized by a point process which involves the de Finetti measure of the increments of the random walk. 
A limit theorem as $N\to \infty$ is found for level set sums of the Gaussian free field.
 In the limit Gaussian process
the covariance function is a mixture of a bivariate normal density, with the correlation mixed by a distribution on $[-1,1]$.
We also study a complex Gaussian field which is the transform of the Gaussian process limit.\\
\begin{keywords} 
\phantom{;}de Finetti sequence; Gaussian Field; Long range random walk; Random walk on a Hypercube.\\[0.1cm]
\end{keywords}\\[0.1cm]
\today
\end{abstract}
\end{frontmatter}
\section{Introduction}\label{Intro:sec}
 A Gaussian free field is defined as a collection of zero mean Gaussian random variables, denoted by $(g_x)_{x \in \cal{V}}$,  for some vertex set $\cal{V}$,
 such that any finite sub-collection of these  random variables is jointly distributed multivariate normal. The covariance structure of the field is defined by the Green function of a  Markov Chain \citep{D1984}. 
For an introduction to Gaussian fields see \cite{WP2020} and \cite{B2016}.
In this paper we study Gaussian free fields on the hypercubes
${\cal V}_N:= \{0,1\}^N$ and ${\cal V}_\infty := \{0,1\}^\infty$.  In our context, the covariance structure is associated to a random walk on the hypercube killed at a geometric stopping time.
\cite{CG2021} studies a class ${\cal G}$ of reversible long-range random walks $ (X_t)$, in discrete time, on the hypercube ${\cal V}_N$, which have an $N$-product Bernoulli$(p)$ stationary distribution,  where $p\ge 1/2$. In this paper we fix $p=\frac{1}{2}$. A description of such processes is the following.  They are  homogeneous in time. There are \iid  random variables $(Z_t)$ in ${\cal V}_N$ with $Z_t$ independent of $(X_\tau)_{\tau \leq t}$ 
such that 
\begin{equation}
X_{t+1} = X_t + Z_t \mod 2,\ t\in \mathbb{N},
\label{rw:000}
\end{equation}
beginning with $X_0$.
Notice that the random walk (\ref{rw:000}) can be taken directly on ${\cal V}_\infty$ if the probabilistic structure of $(Z_t)$ is well defined. 
 As an example on ${\cal V}_N$, if we choose $Z_t$ to be the uniform random variable over all the $N$ unit vectors, we recover the classical simple random walk on the hypercube. 

Section \ref{randomwalkexamples} illustrates some other possible choices of $Z_t$. The general  process described can have  multiple entries of $Z_t$ changing in a single transition.

Our Gaussian free field $\big (g_x\big )_{x\in {\cal V}_N}$ is chosen to have a covariance function
\begin{equation*}
\text{\rm Cov}\big (g_x,g_y\big ) = (1-\alpha)G(x,y;\alpha),
\label {covdefn:00}
\end{equation*}
where $G(x,y;\alpha)$ is the Green function of $(X_t)$, killed at a Geometric time in $\mathbb{N}$, with mean  $\alpha/(1-\alpha)$, for some $\alpha \in (0,1)$.
Usually the covariance matrix is taken as $G(x,y;\alpha)$, however we want to keep the interpretation
\[
\mathbb{P}\big (X_t\text{~is~killed~at~}y\text{~for~}t > 0\mid X_0=x\big )
= (1-\alpha)G(x,y;\alpha).
\]
The random walk (1) is a version of an Ehrenfest urn process where there are $N$ labeled balls which have colours either red or blue. In a transition a ball is chosen at random and its colour is changed. Then entries of $X_t$ which are $1$ indicate red balls and $0$ blue balls.
An extension is to allow multiple balls to be chosen in a transition, controlled by the entries of $Z_t$ which are $1$. \cite{DG2012} explore these processes when $Z_t$ has exchangeable 
entries. For these processes the number of red balls $\|X_t\|$ has transition functions which have a diagonal spectral expansion with eigenfunctions the Krawtchouk polynomials.
More than this fact is a characterization of \emph{all} such processes with Krawtchouk polynomial eigenfunctions. That is, a characterization of all possible eigenvalues corresponding to
Krawtchouk polynomial eigenfunctions. The eigenvalues have a representation as mixtures of Krawtchouk polynomials. The Krawtchouk polynomials always appear as eigenfunctions, and all the information about $Z_t$ is contained in the eigenvalues. \cite{CG2021} consider mixing in such processes and \cite{CG2025} construct a Gaussian field from them. In this paper the entries of $Z_t$ do not need to be exchangeable, though interesting facts emerge when they are. A characterization of these processes is made by characterizing eigenvalues in a spectral expansion, generalizing the Krawtchouk polynomial case. The Gaussian fields we consider are very much controlled by the form of the eigenvalues in the spectral expansions.


Although the Green function is calculated from a discrete-time process it is equivalent to a construction from a continuous-time process. Consider a Poisson process embedding of the discrete process to form transition functions in continuous time $\tau \geq 0$ by
\begin{equation*}
{\cal P}_\tau(y\mid x) = \sum_{k=0}^\infty e^{-\tau}\tau^k\frac{1}{k!}
\mathbb{P}_k(y\mid x ),
\label{ctstime:00}
\end{equation*}
where we denote $\mathbb{P}_k(y\mid x )=\mathbb{P}\big (X_k=y\mid X_0=x)$.
Take $\tau$ in the continuous process to be a killing time with rate  $\theta > 0$. A calculation shows that
\begin{equation*}
\theta u^\theta(x,y) := \theta\int_0^\infty e^{-\theta\tau}{\cal P}_\tau(y\mid x)d\tau
= (1-\alpha) G(x,y;\alpha),\text{~with~}\alpha = \frac{1}{1+\theta}.
\label{Greencts:000}
\end{equation*}
$u^\theta(x,y)$ is a $\theta$-potential kernel occuring in local time calculations in Gaussian processes \citep{MR2006}.

We use the notation $x[k]$ for the $k^{\text{th}}$ entry of a vector $x$. For example if $X_t\in {\cal V}_N$, then $X_t[k]$ is the $k^{\text{th}}$ entry of $X_t$, for $k=1,\ldots, N$. Another notation used will be that $x(C)$ is a sub-vector of $x$ defined by $x(C) = (x[k];k\in C)$ where $C$ is a subset of the index variables of $x$.
\subsection{Spectral expansions}
The spectral expansion of the distribution of $(X_t)$ is important in finding a representation of $(g_x)$ as a linear form of independent standard normal random variables.
\cite{CG2021} characterizes a class of processes $(X_t)$ with transition functions
 for $x,y\in {\cal V}_N$ of
 \begin{equation}
\mathbb{P}(y\mid x) = \frac{1}{2^N}\Big \{1 + \sum_{A\subseteq [N], A \ne \emptyset}
 \rho_A\prod_{k\in A}(-1)^{x[k]+y[k]}\Big \},
 \label{transition:00}
 \end{equation}
 where $[N]:=\{1,2,\ldots,N\}$ and
$ \rho_A = \mathbb{E}\Big [\prod_{k\in A}(-1)^{Z[k]}\Big ]$. 
 
 The distribution of $Z$ itself has an expansion, for $z \in {\cal V}_N$, of
  \begin{equation*}
\mathbb{P}(Z=z) = \frac{1}{2^N}\Big \{1 + \sum_{A\subseteq [N], A \ne \emptyset}
 \rho_A\prod_{k\in A}(-1)^{z[k]}\Big \},
 \label{zexpansion:00}
 \end{equation*} 
 obtained from $\mathbb{P}(Z=z)=\mathbb{P}(z\mid x_0)$ in (\ref{transition:00}), where $x_0$ has all entries zero.
 We write $Z$ as a generic version of $Z_t$ in (\ref{rw:000}).
 The tensor products in the spectral expansion (\ref{transition:00}) are derived from spins, elements of the set $\otimes_{k=1}^N\big \{1, (-1)^{x[k]}\big \}$. They are orthonormal under a uniform distribution of $x \in {\cal V}_N$.
The random walk  $(X_t)$ defined in (\ref{rw:000}), in the case when $(Z_t)$ are 
\iid,  admits the spectral expansion  appearing in (\ref{transition:00}). Conversely such a spectral expansion always arises from a random walk with \iid $(Z_t)$.
 Let $x(A) = \big \{x[k]; k \in A\big \}$, for $x\in {\cal V}_N$, $A\subseteq [N]$, then
$x(A)$ is of dimension $|A|$.  
Another way to write (\ref{transition:00}) in terms of Hamming distances is
 \begin{equation*}
\mathbb{P}(y\mid x) = \frac{1}{2^N}\Big \{1 + \sum_{A\subseteq [N], A \ne \emptyset}
 \rho_A \cdot(-1)^{\|y(A)-x(A)\|}\Big \}.
 \label{transition:0001}
 \end{equation*}
If the entries in $Z$ have an exchangeable distribution  then (\ref{transition:00}) simplifies to an expansion in terms of Krawtchouk polynomials. With exchangeability
$\rho_A$ depends only on $|A|$, and we write $\rho_{|A|} \equiv \rho_A$. In this case
\begin{align}
\mathbb{P}(y\mid x) &= \frac{1}{2^N}\Big \{1 + \sum_{k=1}^N\rho_k{N\choose k}
Q_k(\|y-x\|;N,\frac{1}{2})\Big \},
 \label{transition:01}\\
  \rho_k &= \mathbb{E}\big [(-1)^{Z[1]+\cdots + Z[k]}\big ]\label{transitionspin:02}\\  
  &= \mathbb{E}\big [Q_k(Z[1]+\cdots + Z[k];N,\frac{1}{2}) \big ],
 \nonumber
\end{align}
where $(Q_k)$ are Krawtchouk polynomials orthogonal on a Binomial$(N,\frac{1}{2})$ distribution. A generating function for them is
\begin{equation}
\sum_{k=0}^N{N\choose k}Q_k(w;N,\frac{1}{2})\phi^k = (1-\phi)^w(1+\phi)^{N-w}.
\label{kgf:001}
\end{equation}
This version of the Krawtchouk polynomials is scaled so that $Q_k(0)=1$.
Another representation of $(Q_k)$ in terms of spins, for $w\in {\cal V}_N$, $k \in [N]$, is
\begin{equation}
Q_k\big (\|w\|;,N,\frac{1}{2}\big ) = {N\choose k}^{-1}\sum_{A\subseteq [N],\ |A|=k}\prod_{j\in A}\ (-1)^{w[j]},
\label{Qspin:00}
\end{equation}
with $Q_0\big (\|w\|;N,\frac{1}{2}\big ) = 1$. See \cite{DG2012} for (\ref{kgf:001}) and (\ref{Qspin:00}).

The $t$-step transition probabilities of $(X_t)$, $\mathbb{P}_t(y\mid x)$, where $X_t=y$ and $X_0=x$, are similar to  
(\ref{transition:00}) and (\ref{transition:01}) with $\rho_A$ and $\rho_k$ replaced by
$\rho_A^t$ and $\rho_k^t$, respectively. {$\mathbb{P}_0(y\mid x)=\delta_{xy}$ is correctly specified by taking $t=0$ in the spectral expansion for $\mathbb{P}_t(y\mid x)$. Provided $\sum_{k\in A}Z[k]$ is not constant for all $A\subseteq [N]$ then $\lim_{t\to \infty}\mathbb{P}_t(y\mid x) = \frac{1}{2^N}$, that is the chain is ergodic and the steady state distribution is uniform over the hypercube. The stationary distribution has an intuitive explanation that whenever an entry in $X_t$ is chosen
 with $Z_t[k]=1$ it flips, and the entries in $Z_t$ control how lazy the process is because if $Z_t[k]=0$ then $X_{t+1}[k] = X_t[k]$.

The Hamming distance from zero, $\|X_t\|$, is the same as a generalized Ehrenfest urn process studied in 
\cite{DG2012}. The process $(\|X_t\|)$  is a Markov chain with transition probabilities for $u,w \in [N]$, $t\in \mathbb{N}$ of
\begin{equation*}
\mathbb{P}(\|X_{t+1}\|=w \mid \|X_t\|=u)
= {N\choose w}\frac{1}{2^N}\Big \{1 + \sum_{k=1}^N\rho_k{N\choose k}
Q_k(w;N,\frac{1}{2})Q_k(u;N,\frac{1}{2})\Big \}.
\label{transition:02}
\end{equation*}
Let $T_\alpha$ be a geometric random variable with
\begin{equation}
\mathbb{P}\big (T_\alpha = t) = (1-\alpha)\alpha^t,\ t\in \mathbb{N}.
\label{killing:00}
\end{equation}
The Green function of the stopped process $X_{T_{\alpha}}$ corresponding to killing at rate $1-\alpha$ in each transition satisfies
\begin{align}
(1-\alpha)G(x,y;\alpha)	
& = \mathbb{E}\big [\mathbb{P}_{T_{\alpha}}\big (y\mid x)\big ]\nonumber \\
&= \sum_{t=0}^\infty(1-\alpha)\alpha^t\mathbb{P}_t\big (y\mid x)\nonumber \\
& = \frac{1}{2^N}\Big \{1 + \sum_{A\subseteq [N], A \ne \emptyset}
 \frac{1}{1 + \frac{\alpha}{1-\alpha}(1-\rho_A)}
 \prod_{k\in A}(-1)^{x[k]+y[k]}\Big \}
 \label{transition:a00}\\
 & = \frac{1}{2^N}\Big \{1 + \sum_{k=1}^N
 \frac{1}{1 + \frac{\alpha}{1-\alpha}(1-\rho_k)}
 {N\choose k}Q_k(\|y-x\|;N,\frac{1}{2})\Big \},
 \label{transition:a01}
\end{align}
where (\ref{transition:a00}) holds in general and (\ref{transition:a01}) holds when the entries in $Z$ are exchangeable. $G(x,y;\alpha)$ is the mean number of visits to $y$ beginning from $x$. $(1-\alpha)G(x,y;\alpha)$ has an interpretation as the probability of being killed at $y$ beginning from $X_0=x$.
The Hamming distance of the stopped process $\|X_{T_{\alpha}}\|$ has a Green function satisfying
\begin{align*}
&(1-\alpha)G_H(\|x\|,\|y\|;\alpha)		\nonumber \\
& = {N\choose \|y\|}\frac{1}{2^N}\Big \{1 + \sum_{k=1}^N
 \frac{1}{1 + \frac{\alpha}{1-\alpha}(1-\rho_k)}
 {N\choose k}Q_k(\|y\|;N,\frac{1}{2})Q_k(\|x\|;N,\frac{1}{2})\Big \}.
 \label{transition:b01}
\end{align*}
The process $(\|X_t\|)$ is equivalent to the generalized Ehrenfest urn process in \cite{DG2012}.

We suppress parameters in  notation but use them when important. For example $(X_t)$ is the suppressed version of $(X^{(N)})_{t\in {\cal V}_N}$ and $G(x,y;\alpha)$
is the suppressed version of 
$G^{(N)}(x,y;\alpha_N)$. 
\section{Summary of paper contents} 
Section \ref{randomwalkexamples} contains examples (a)-(g) of random walks on ${\cal V}_N$ and ${\cal V}_\infty$ with geometric killing. In the examples $\rho_A$ is calculated from knowing the distribution of $Z\in {\cal V}_N$.
In Example (a) the distribution of $Z$ is general; in example (b) the $N$ entries in $Z$ are from a $0-1$ de Finetti sequence; in example (c) $Z$ is a unit vector from a 
simple random walk. An extension made to when exactly $M$ entries of $Z$ are unity with the reminder zero; example (d) considers a limit when  $N\to \infty$, $\alpha \to 1$, $\rho_A \to 0$ with $\rho_A/(1-\alpha)$ converging and makes a comparision of the limits between the two random walks in (c); in example (e) one entry in $Z$ is chosen at random to be Bernoulli$(1/2)$
and the other entries are zero; in example (f) a de Finetti measure is taken as Beta$(a_\alpha,\kappa)$ and the limit considered when $\alpha \to 1$ and $a_\alpha/(1-\alpha) \to 0$. The limit has a description in terms of a Poisson Dirichlet point process; in Example (g) the entries of $Z$ form a homogeneous Markov chain.

In Section \ref{Gauss:sec} a strong representation of $\big (g_x\big )$ 
as a linear sum of \iid  standard normal random variables $\big (g_A\big )_{A\subseteq [N]}$ is found.
In Proposition \ref{Markov:000} Gaussian fields $\big (g_{x_N}\big )_{x_N \in {\cal V}_N}$, 
with $x_N$ the first $N$ entries in $x\in {\cal V}_\infty$, are coupled over $N\in \mathbb{N}$ and shown to form a Markov chain.
If $Z$ has exchangeable entries, then $g_x$ is a linear sum 
 of $k$-spin variables $S_k\big (x;[N],(g_A)\big )$, which  are defined, for $C$ a finite subset of $\mathbb{N}$, as
\begin{equation}
S_k(x;C,(g_A)) =\sum_{A \subseteq C,\ |A|=k}\ \ \prod_{j\in A}(-1)^{{x}[j]}g_A,
 \label{kspin:123}
\end{equation}
with $(g_A)_{A\subseteq C}$ \iid N$(0,1)$ random variables.
$S_k\big (x;[N],(g_A)\big )$ are $k$-spin Hamiltonian models in spin-glasses \citep{T2014}. If $k=2$, (\ref{kspin:123}) is the Sherrington-Kirkpatrick model.
In Subsection \ref{GaussCentre:sec} a centered Gaussian field is defined by
$g^\circ_x = g_x - 2^{-N}\sum_{y\in {\cal V}_N}g_y$,
then Gaussian fields $g_{x(C_x)}$ and  $g^\circ_{x(C_x)}$ are defined by the marginal sums of Gaussian variables $g_x, g^\circ_x,\
 x \in {\cal V}_N$ with fixed entries $x[k],\ k \in C_x \subseteq [N]$.
If
$Z$ has exchangeable entries then $g_{x(C_x)},g^\circ_{x(C_x)}$ are linear sums of $k$-spin variables (\ref{kspin:123}) with $C=C_x$.
We take the weak limit from  $\big (g^{\circ}_{x(C_x)}\big )$
as $N\to \infty$ resulting in a Gaussian field 
$\big (g^{\infty}_{x(C_x)}\big )$
 where $x \in {\cal V}_\infty$ and $C_x$ is a finite subset of $\mathbb{Z}_+$.
The Gaussian field has independent increments in the sense that $g^\infty_{x(C_x)}$ and $g^\infty_{y(C_y)}$ are independent if $C_x\cap C_y = \emptyset$.
Section \ref{deF:sec} studies an important model where the elements of $Z$ arise from $N$ random variables taken from a de Finetti sequence with elements in $\{0,1\}$. 
Although the Gaussian field is on ${\cal V}_N$ it is effectively controlled by the de Finetti measure of the sequence. The elements of $X_t$ are $N$ entries from a de Finetti sequence as well. We investigate how the de Finetti measure for $X_t$ changes over time. The de Finetti measure of $X_{T_\alpha}$ when there is killing is calculated. In this context $\alpha$ does not usually depend on $N$. An important random variable is $Y$, the product of points in a point process with points in $[-1,1]$.  The probability generating functional (\pgfl) for the point process is
\begin{equation*}
G[f] = \frac{1}{1 + \frac{\alpha}{1-\alpha}\Big (
\int_\spin \big (1-f(\xi)\big )\nu^\spin(d\xi)\Big )},
\label{pgfl:123}
\end{equation*}
where $\nu^\spin$ is a shifted measure of the de Finetti sequence measure $\nu$ on $[0,1]$. If $\omega$ is a random variable with probability measure $\nu$, then $\nu^\spin$ is the probability measure of $1-2\omega$.
In Proposition \ref{propde:90} it is shown that
\begin{equation}
\mathbb{E}\big [Y^k\big ] = \frac{1}{1 + \frac{\alpha}{1-\alpha}(1-\rho_k)},
\label{Y:66}
\end{equation}
where $\rho_k$ is the mean spin (\ref{transitionspin:02}).
Then
\begin{equation}
(1-\alpha) G(x,y;\alpha) = \frac{1}{2^N}\Big \{1 + \sum_{k=1}^N\mathbb{E}\big [Y^k\big ]{N\choose k}Q_k(\|y-x\|;N,\frac{1}{2})\Big \},
\label{yg:000}
\end{equation}
by substituting (\ref{Y:66}) in (\ref{transition:a01}).
 From (\ref{yg:000}) and (\ref{kgf:001}) in this model,
\begin{equation}
(1-\alpha)G(x,y;\alpha)	
=  \mathbb{E}\Big [ 
\Big (\frac{1}{2}-\frac{Y}{2}\Big )^{\|y-x\|}\Big (\frac{1}{2}+\frac{Y}{2}\Big )^{N-\|y-x\|}\Big ].
\label{introeigen:00}
\end{equation}
Proposition \ref{propkilled:00} finds the joint Laplace transform of $-\log |Y|$ and the probability of a positive/negative sign of $Y$. If $\nu^\spin(d\xi)$ is symmetric about 
$\xi=0$, (\ $\nu(d\omega)$ is symmetric about $\omega = \frac{1}{2}$), then the sign of $Y$ is independent of $|Y|$ and equal to $\pm 1$ with probability $\frac{1}{2}$. 
In Example \ref{Betaexample:00} we work through a model with $\nu^\spin$ symmetric about zero when $2\nu^\spin$ is a Beta$(a,b)$ measure on $(0,1)$ for $\xi > 0$.
In Proposition \ref{deF:340} a strong representation (\ref{exchcirc:00ab})
is found for the Gaussian field $(g_x)$. $G[f]$ is infinitely divisible. The product of points in a point process with \pgfl $G[f]^{\frac{1}{2}}$, with notation $Y_{\frac{1}{2}}$, and the spin variables (\ref{kspin:123}) 
 are important in the representation.
In the de Finetti model the Laplace transform of $-\log |Y|$ is shown to be
\begin{equation*}
 \frac{1}{1 + \frac{\alpha}{1-\alpha}\Big (
\int_\spin \big (1-|\xi|^\theta\big )\nu^\spin(d\xi)\Big )}.
\label{intro:55}
\end{equation*}
In Subsection \ref{subsection:level} 
level set sums from the origin 
$
\vartheta_v^{(N)}
= \sum_{y:\|y\|=v}g_y^{(N)},\ v\in \mathbb{R}_+,
$
are studied.
In a central limit theorem for the index, $\sqrt{N}2^{-(1+\frac{N}{2})}\vartheta^{(N)}_{v_N}$ with $v_N = \lfloor \frac{N}{2}+ \frac{\sqrt{N}}{2}t\rfloor$
is shown to converge weakly as $N\to \infty$ to a Gaussian process
$(\varkappa_t)_{t\in \mathbb{R}}$ where the covariance function is a bivariate normal density with the correlation mixed by a distribution on $[-1,1]$. 
A strong representation of $\varkappa_t$ is found. A complex Gaussian field $\widehat{\varkappa}_\theta = \int_{\mathbb{R}} e^{i\theta t}\varkappa_t dt$ is defined. The covariance function of $\widehat{\varkappa}_\theta$ and a strong representation are found. The covariance function is much simpler than that of $\varkappa_t $.
There is an inversion formula 
$
\varkappa_t = \frac{1}{2\pi}\int_{\mathbb{R}} e^{-i\theta t}\widehat{\varkappa}_\theta d\theta
$.
 The limit results and the complex transform extend those of \cite{CG2025} for the simple hypercube random walk.
 \section{Random walk examples} \label{randomwalkexamples}
\begin{enumerate}[label=(\alph*)]
\item[]
\item \label{item:00}
Let $(Z[k], k \in [N])$ be a set of random variables in $\{0,1\}$ from a collection
$(Z[k], k \in \mathbb{Z}_+)$ not depending on $N$, and also suppose that $\alpha$ does not depend on $N$.
Denote
\begin{equation*}
b_A = \frac{\alpha}{1-\alpha}(1-\rho_A),\quad\rho_A = \mathbb{E}\Big [\prod_{j\in A}(-1)^{Z[j]}\Big ].
\label{bAdef:00}
\end{equation*}
If $b_A$ only depends on $|A|$ write $b_{|A|} \equiv b_A$.
It is useful to include $\rho_\emptyset = 1,\ b_\emptyset = 0$.
\item \label{item:01}
 \ref{item:00} is naturally satisfied by a de Finetti sequence with measure $\nu$.
Then, with $|A|=k$, 
\begin{align}
\rho_k &= \mathbb{E}\big [(-1)^{Z[1]+\cdots +Z[k]}\big ]\nonumber \\
& = \int_{[0,1]}(1-2\omega)^k\nu(d\omega)
\label{deF:101}\\
b_A &= \frac{\alpha}{1-\alpha}\int_{[0,1]}\big (1 - (1-2\omega)^k\big )\nu(d\omega).
\nonumber
\end{align}
The integral calculation in (\ref{deF:101}) follows from noting that the \pgf of a Binomial$(k,\omega)$ random variable is $(1+\omega(s-1))^k$, then choosing $s=-1$ to obtain
$(1-2\omega)^k$.\\
In a particular case when $\nu = \delta_{\{p\}}$ the entries of $Z$ are \iid  Bernoulli$(p)$ and 
\[
b_A =  \frac{\alpha}{1-\alpha}\big (1 - (1-2p)^k\big ).
\]
\item \label{item:02} In the simple random walk 
by symmetry
\begin{align*}
\rho_A &= \mathbb{E}\big [(-1)^{Z[1]+\cdots + Z[|A|]}\big ] \nonumber\\
 &= -1 \times \frac{|A|}{N} + 1 \times \Big ( 1 -  \frac{|A|}{N}\Big )
= 1 - \frac{2|A|}{N}.
\label{simplecalc:00}
\end{align*}
This calculation can be easily extended to a model where $M$ of $N$ entries in $Z$ are unity and $N-M$ are zero. If $M=2$
\[
\rho_A = 1 - \frac{2|A|(N-|A|)}{N(N-1)}.
\]
For $M\geq 1$
\begin{align}
\rho_A &= \sum_{j\geq 0}(-1)^j
\frac{
{M\choose j}{N-M\choose |A| -j}
}
{
{N\choose |A|}
}\label{Mcase:00}\\
&= Q_M\big (|A|;N,\frac{1}{2}\big ) = Q_{|A|}\big (M;N,\frac{1}{2}\big ).
\label{Mcase:01}
\end{align}
Equation (\ref{Mcase:00}) is found by considering the number of entries $j$ from $A$ which fall in the $M$ entries which are equal to unity. Equation (\ref{Mcase:01}) follows from the form of the generating function
(\ref{kgf:001}) as a convolution.
If $M=\lfloor \omega N\rfloor $,
 $Q_{|A|}\big (M;N,\frac{1}{2}\big ) \to \big ( 1 - 2\omega \big )^{|A|}$ as $N \to \infty$ .
 In the limit the entries of $Z$ are effectively \iid Bernoulli$(\omega)$.
 A limit is
 \[
\lim_{N\to \infty} b_A = \frac{\alpha}{1-\alpha}\big (1 - (1- 2\omega)^{|A|}\big ).
\]

\item \label{item:03}
In the simple random walk take $\alpha_N= 1 - \frac{\gamma}{N}$. 
Then
\[
b_A = \lim_{N\to \infty}\frac{1-\frac{\gamma}{N}}{\frac{\gamma}{N}}\frac{2|A|}{N}
= \frac{2|A|}{\gamma}.
\]
The same limit holds when two entries are chosen at random to change.

\cite{CG2025} use the idea of the killing rate tending to zero  to obtain a limit theorem for level set sums in a Gaussian field.  In this example the limit Gaussian field is well defined on ${\cal V}_\infty$ but the random walk is not well defined on ${\cal V}_\infty$.
\item \label{item:04}
$Z$ can be chosen so that in each transition one entry in $X_{t+1}$ becomes Bernoulli$(\frac{1}{2})$ distributed and independent of the other entries. To achieve this choose $K_{t+1}\in [N]$ and let 
$
\mathbb{P}\big (Z[K_{t+1}]=0\big)= \frac{1}{2},\  \mathbb{P}\big (Z[K_{t+1}]=1\big)= \frac{1}{2}$.
Then in a transition
\begin{align*}
&\mathbb{P}\big (X_{t+1}[K_{t+1}]=0\big ) = 
\mathbb{P}\big (X_{t}[K_{t+1}]=0\big )\mathbb{P}\big (Z[K_{t+1}]=0\big)\\
&~~~~~~~~~~~~~~~~~~~~~~~~~+ \mathbb{P}\big (X_{t}[K_{t+1}]=1\big ) \mathbb{P}\big (Z[K_{t+1}]=1\big) = \frac{1}{2}
\end{align*}
and $X_{t+1}[j] = X_t[j]$ if $j\ne K_{t+1}$. $K_{t+1}$ is a random index chosen uniformly from $[N]$ with replacement over time.
The probability that exactly $t$ transitions are needed so $X_t$ has $N$ independent Bernoulli$(\frac{1}{2})$ entries is the classical coupon collectors problem, with solution 
\[
\frac{{\begin{Bmatrix} t-1\\[-3.25pt]N-1 \end{Bmatrix} N!}}{N^t}
\]
where ${\scriptsize \begin{Bmatrix} a\\[-3.25pt] b \end{Bmatrix}}$, $b\geq a \in \mathbb{N}$ are Stirling numbers of the second kind; which count the number of ways to partition a set of $a$ labelled objects into $b$ nonempty unlabelled subsets.  An explicit form is
\[
\begin{Bmatrix} a\\[-3.25pt] b \end{Bmatrix} =
\sum_{j=0}^b\frac{(-1)^{b-j}j^a}{(b-j)!j!}.
\]
 Using a generic notation for $Z,K$, 
\[
\prod_{j\in A}(-1)^{Z[j]} =
\begin{cases}
1&K \not\in A,\\
-1&K \in A
\end{cases}
\]
Therefore
\[
\rho_A = -\frac{|A|}{N} + 1 - \frac{|A|}{N} = 1 - 2\frac{|A|}{N}.
\]
The mechanism is different from Example (c), but the probability distribution of $(X_t)$ is the same.
\item \label{item:05} In a de Finetti sequence with a measure $\nu_\alpha$  suppose $\alpha \to 1$ and $\nu_\alpha \to \delta_0$ such that $\frac{\alpha}{1-\alpha}\nu_\alpha \to \nu_*$, which is a non-negative measure.
An Example is when $\nu_\alpha$ is a Beta$(a_\alpha,\kappa)$ measure with $a_\alpha \to 0$ and
$\frac{a_\alpha}{1-\alpha}\to \kappa < \infty$. Then
\begin{align}
\frac{\alpha}{1-\alpha}\nu_\alpha(d\omega) &= 
\frac{\alpha}{1-\alpha}\frac{\Gamma(a_\alpha+\kappa)}{\Gamma(a_\alpha)\Gamma(\kappa)}
\omega^{a_\alpha-1}(1-\omega)^{\kappa-1}d\omega\nonumber \\
&= \frac{\alpha a_\alpha}{1-\alpha}\frac{\Gamma(a_\alpha+\kappa)}{\Gamma(a_\alpha+1)\Gamma(\kappa)}
\omega^{a_\alpha-1}(1-\omega)^{\kappa-1}d\omega\nonumber \\
& \to \kappa \omega^{-1}(1-\omega)^{\kappa-1}d\omega.
\label{limita:00}
\end{align}
$\nu_*$ is not a finite measure because of the singularity at zero, but this does not affect
the limit form
\begin{align}
b_A &= \kappa\int_0^1(1-(1-2\omega)^{|A|})\omega^{-1}(1-\omega)^{\kappa -1}d\omega\nonumber\\
&= \kappa\sum_{j=1}^{|A|}(-1)^{j+1}\frac{2^j}{j}\cdot \frac{|A|_{[j]}}{\kappa_{(j)}}.
\nonumber
\end{align}
Notation is that for $a\in \mathbb{R}$, $j \in \mathbb{N}$, 
$a_{[j]} = a(a-1)\cdots (a-j+1)$, $a_{(j)} = a(a+1)\cdots (a+j-1)$.
Equation (\ref{limita:00}) is the first moment measure in a Poisson Dirichlet point process ${\cal PD}(\kappa)$. $b_A$ is the mean number of points in a sample of $|A|$ from the points which have an odd number of representatives. Denoting the points by $(u_{(j)})$,
$\sum_{j=1}^\infty u_{(j)} = 1$ and
 $b_A = \mathbb{E}\Big [\sum_{j=1}^\infty \Big ( 1 - (1-2u_{(j)})^{|A|}\big )\Big ]$.
 \item \label{item:06}
 An example where the entries of $Z$ are not exchangeable is when the entries of $Z$ form a homogeneous Markov chain starting with $Z[1]$. 
A calculation of $\rho_A$ is straightforward once the elements of $A$ are ranked $a_{|A|} > a_{|A|-1} > \cdots > a_1$.
\end{enumerate}
\section{Gaussian free fields on a hypercube}\label{Gauss:sec}
Let $(g_x)$ be a zero mean Gaussian field indexed by $x\in {\cal V}_N$ with covariance matrix 
\begin{equation*} 
\text{Cov}\big (g_x,g_y\big ) = (1-\alpha)G(x,y;\alpha),\ x,y \in {\cal V}_N,
\label{definition:00}
\end{equation*}
where the Green function is from the killed process $X_{T_\alpha}$.
Denote 
$\big (G(x,y;\alpha)\big )_{ x,y\in {\cal V}_N  }$ 
as a 
${\cal V}_N\times {\cal V}_N$ matrix with elements $G(x,y;\alpha)$.
The reversible form of the spectral expansion (\ref{transition:a00}) implies that $(1-\alpha)\big (G(x,y;\alpha)\Big )_{x,y\in {\cal V}_N}$ is a positive definite matrix. The Gaussian field $(g_x)$ arising from a simple random walk where $\rho_A = 1 - \frac{2|A|}{N}$ is studied in \cite{CG2025}.

The Gaussian field studied in this paper can be viewed as a discrete GFF constructed from the discrete time process (\ref{rw:000}) with killing at rate $1-\alpha$ in each transition. Suppose a cemetery state $\Delta$ is added to the state space
${\cal V}_N$ to form ${\cal V}_N^+ = {\cal V}_N\cup \partial$, where $\partial = \{\Delta\}$ is now the boundary of the state space ${\cal V}_N^+$. 
Then a transition function including $\partial$ is 
$\mathbb{P}^+\big (y\mid x\big )$, $x,y\in {\cal V}_N^+$ where
$\mathbb{P}^+\big (y\mid x\big )=\alpha\mathbb{P}\big (y\mid x\big )$, $x,y \in {\cal V}_N$,
$\mathbb{P}^+\big (\Delta\mid\Delta\big ) = 1$,  $\mathbb{P}^+\big (\Delta \mid x\big ) = 1-\alpha$, $\mathbb{P}^+\big (x\mid \Delta\big )=0$
 for $x\in {\cal V}_N$. Denote $G^{(-1)}\big (x,y;\alpha\big )$ as the
 $(x,y)^{\text{th}}$ element in the matrix inverse
$ \big (G(x,y;\alpha)\big )^{-1}_{ x,y\in {\cal V}_N  }$.
The Dirichlet energy in this context is
\[
 \frac{1}{2}
 \sum_{x,y \in {\cal V}_N^+}\mathbb{P}^+(y\mid x)(g_x-g_y)^2.
 \]
 The next Proposition is a discrete time version of Theorem 1.8 in \cite{B2016} for a general discrete Gaussian free field arising from a continuous time process.
 \begin{Proposition}\label{GFFEnergy}
 The law of the Gaussian Free Field, with covariance matrix 
 $ \big (G(x,y;\alpha)\big )_{ x,y\in {\cal V}_N  }$,
 has a density proportional to
 \begin{align}
 &\exp \Big \{-\frac{1}{4(1-\alpha)}
 \sum_{x,y \in {\cal V}_N^+}\mathbb{P}^+(y\mid x)(g_x-g_y)^2 
 - \frac{1}{4}\sum_{x\in {\cal V}^+_N}g_x^2
 \Big \}\nonumber\\
 &=  \exp \Big \{ -\frac{1}{2(1-\alpha)}\sum_{x,y\in {\cal V}_N}
 G^{(-1)}\big (x,y;\alpha\big )
 g_xg_y\Big\},
 \label{GFF:00}
 \end{align}
 where we set $g_\Delta=0$.
 \end{Proposition}
 \begin{proof}
 We simplify the Dirichlet energy.
  \begin{align}
 &\frac{1}{2}
 \sum_{x,y \in {\cal V}_N^+}\mathbb{P}^+(y\mid x)(g_x-g_y)^2\nonumber\\
 &=\frac{1}{2}
\sum_{x,y \in {\cal V}_N}\alpha\mathbb{P}(y\mid x)(g_x-g_y)^2+\frac{1}{2}(1-\alpha)\sum_{x\in {\cal V}_N}g_x^2\nonumber\\
 &= -\frac{1}{2}\sum_{x, y \in {\cal V}_N}
 \Big (\delta_{xy} - \alpha\mathbb{P}(y\mid x)\Big )(g_x-g_y)^2+\frac{1}{2}(1-\alpha)\sum_{x\in {\cal V}_N}g_x^2\nonumber\\
 &= -\frac{1}{2}\sum_{x,y \in {\cal V}_N}
 G^{(-1)}\big (x,y;\alpha\big )(g_x-g_y)^2
 +\frac{1}{2}(1-\alpha)\sum_{x\in {\cal V}_N}g_x^2\label{GFF:01}\\
 &=  \sum_{x,y\in {\cal V}_N,\ x \ne y}G^{(-1)}\big (x,y;\alpha\big )g_xg_y
 - \sum_{x\in {\cal V}_N}\Big (1-\alpha
-\frac{1}{2}(1-\alpha)\Big )g_x^2\nonumber\\
 &=  \sum_{x,y \in {\cal V}_N,\ x\ne y}G^{(-1)}\big (x,y;\alpha\big )g_xg_y
 - \frac{1}{2}(1-\alpha)\sum_{x\in {\cal V}_N}g_x^2
 \label{GFF:02}
 \end{align}
 Line (\ref{GFF:01}) follows because $\Big (I - \alpha\mathbb{P}(y\mid x)\Big )_{x,y\in {\cal V}_N}$ and 
 $ \Big (G\big (x,y;\alpha\big )\Big )_{x,y\in {\cal V}_N}$ have the same eigenvectors 
 \begin{equation}
 \frac{1}{2^{N/2}} \prod_{k\in A}(-1)^{x[k]},\ A \subseteq [N],\ x\in {\cal V}_N,
 \label{ev:33}
 \end{equation}
 and respective eigenvalues 
 \begin{equation*}
 1-\alpha \rho_A,\ \frac{1}{1 - \alpha\rho_A},\text{~where~} A\subseteq [N].
 \end{equation*}
 This is related to Proposition 1.2 in \cite{B2016}. 
 Denote the matrix $P^+ = \big (\mathbb{P}^+(y\mid x)\big )_{x,y \in {\cal V}_d^+}$.
 The rate matrix $\widehat{Q}=P^+- I$ restricted to ${\cal V}_d\times {\cal V}_d$ satisfies 
 $(-\widehat{Q})^{-1}_{x y}=G(x,y;\alpha)$, $x,y \in {\cal V}_d$.\\
 When $A=\emptyset$ the eigenvector is  
 $2^{-N/2}, x \in {\cal V}_N$ and $\rho_\emptyset = 1$.
 Multiply (\ref{GFF:02}) by $-\frac{1}{2(1-\alpha)}$ to obtain (\ref{GFF:00}),
 noting cancellation with the term $\frac{1}{4}\sum_{x\in {\cal V}_N}g_x^2$ to obtain the second line.
 \end{proof}
 Knowing a diagonal eigenfunction expansion for $(1-\alpha)G(x,y;\alpha)$ it is straightforward to obtain a strong decomposition of $(g_x)$ by checking the covariance matrix. This is analogous to Theorem 1.7 in \cite{B2016}. In their expansion
\[
G(x,y) = \sum_{m=1}^N\frac{1}{\lambda_m}e_m(x)e_m(y),
\]
with $x,y$ in the appropriate space, our analog of 
$\lambda_m^{-1}$ is $(1+\frac{\alpha}{1-\alpha}\rho_A)^{-1}$, and $e_m(x)$ has an analog of (\ref{ev:33})  for $A\subseteq [N]$. 
\begin{Proposition}\label{prop:0}
A strong decomposition of $(g_x)$ into a linear form of independent N$(0,1)$ random variables $(g_A)_{A \subseteq [N]}$ is, 
for all $x\in {\cal V}_N$,
\begin{equation}
g_x = \frac{1}{2^{N/2}}\Big \{g_\emptyset 
+ \sum_{A\subseteq [N], A \ne \emptyset} \frac{1}{\sqrt{1 + \frac{\alpha}{1-\alpha}(1-\rho_A)}}
\cdot \prod_{j\in A}(-1)^{x[j]}\cdot g_A\Big \}
 \label{decomposition:a00}.\\
\end{equation}
If $Z$ has exchangeable entries then
\begin{equation}
g_x
=  \frac{1}{2^{N/2}}\Big \{g_\emptyset +
\sum_{k=1}^{N}
 \frac {1}{\sqrt{1 + \frac{\alpha}{1-\alpha}(1-\rho_k)}}
\cdot S_k\big (x;[N],(g_A)\big )
\Big \}.
\label{exchcirc:00a}
\end{equation}
Recall that $\rho_A = \mathbb{E}\big [\prod_{j\in A}(-1)^{Z[j]}\big ]$
and $S_k\big (x;[N],(g_A)\big )$ is a Gaussian $k$-spin defined in (\ref{kspin:123}).
\end{Proposition}
\begin{proof}
It is straightforward to show that the covariance function of $(g_x)$ agrees with $(1-\alpha)G(x,y)$ 
in (\ref{transition:a00}) by using orthogonality of $(g_A)$. This is enough to show the representation.
 In the exchangeable case $\rho_A = \rho_{|A|}$ and (\ref{exchcirc:00a}) follows by collecting terms in 
(\ref{decomposition:a00}) where $|A|=k$ to form $S_k\big (x;[N],(g_A)\big )$.
\end{proof}
\begin{Corollary} \label{ziidb}
If $Z$ has \iid Bernoulli$(1/2)$ entries, then the random walk $(X_t)$ mixes in one step.
A representation for the Gaussian field is 
\[
g_x 
= \sqrt{1-\alpha}\ g^\ast_x +  \sqrt{ \frac{\alpha}{2^N}}\ g_\emptyset^\ast
\]
where $(g^\ast_x)$  and $g^\ast_\emptyset$ are independent $N(0,1)$ random variables.
\end{Corollary}
\begin{proof}
If the entries of $Z$ are \iid Bernoulli$(1/2)$ then for $A\subseteq [N]$, $A\ne \emptyset$,
\[
\rho_A = \mathbb{E}\big (\prod_{j\in A}(-1)^{Z[j]}\big ) = 0
\]
so for $x,y \in {\cal V}_N$ 
\[
\mathbb{P}\big (X_{1} = y\mid X_0 = x\big ) = \frac{1}{2^N}
\]
from (\ref{transition:00}), showing mixing in one step. The Markov chain is homogeneous so 
\[
\mathbb{P}\big (X_{t+1} = y\mid X_t = x\big ) = \frac{1}{2^N}
\]
for all $t\in \mathbb{N}$.
Therefore
\begin{equation*}
(1-\alpha)G(x,y;\alpha) = (1-\alpha)\delta_{xy}  + \frac{\alpha}{2^N}
\label{temp:465}
\end{equation*}
and the representation for $g_x$ follows.
\end{proof}

\begin{Corollary}\label{symmetrycorr:000}
Let $(\vartheta_{v})$ be a Gaussian field of level sets from the origin, defined by
$\vartheta_{v}
= \sum_{\|x\|=v}g_{x}$, $v = 0,1,\ldots, N$.
 Then in the exchangeable case
\begin{equation}
\vartheta_{v} = {N\choose v}
 \frac{1}{2^{N/2}}\Big \{g_\emptyset +
\sum_{k=1}^{N}
 \frac {1}{\sqrt{1 + \frac{\alpha}{1-\alpha}(1-\rho_k)}}
 \sqrt{{N\choose k}}Q_k\big (v;N,\frac{1}{2}\big )\cdot \zeta_k\Big \},
\label{symmetry:001}
\end{equation}
where 
\begin{equation}
\zeta_k = \frac{1}{\sqrt{{N\choose k}}}\sum_{A\subseteq [N];|A|=k}g_A\
\label{zeta:00}
\end{equation}
are independent N$(0,1)$ random variables. 
\end{Corollary}
\begin{proof}
The representation (\ref{symmetry:001}) for $\vartheta_{v}$ follows from summing in (\ref{decomposition:a00}) and using the Krawtchouk polynomial expression (\ref{Qspin:00}). Consider a particular $A$ with $|A|=k$. Then
\[
\sum_{\|x\|=v}\frac {1}{\sqrt{1 + \frac{\alpha}{1-\alpha}(1-\rho_k)}}\prod_{j\in A}(-1)^{x[j]}\ g_A
= \frac {1}{\sqrt{1 + \frac{\alpha}{1-\alpha}(1-\rho_k)}}
{N\choose v}Q_k\big (v;N,\frac{1}{2}\big )g_A
\]
Now define $\zeta_k$ by (\ref{zeta:00}). There are ${N\choose k}$ \iid normal terms in the sum definition of $\zeta_k$, so it has a variance of 1. The term $\sqrt{{N\choose k}}$ arising in 
the denominator of $\zeta_k$ is accounted for in multiplication of $Q_k\big (v;N,\frac{1}{2}\big )$.
\end{proof}
\begin{Remark}
Marginally $g_x =^{\cal D}g_{x(0)}$, where $x(0)$ has all entries zero and a decomposition similar to (\ref{decomposition:a00}) 
with 
$\prod_{j\in A}(-1)^{x[j]}=1$.
A major point is that $(g_x)$ are coupled over $(g_A)$ so the spin factor cannot be removed in the joint distribution of $(g_x)$.
\end{Remark}
\begin{Remark}
 We know the covariance structure of the spins as
\begin{align}
\text{\rm Cov}\big (S_j\big (x;[N],(g_A)\big ),S_k\big (y;[N],(g_A)\big )\big )
&= \delta_{jk}\sum_{A\subseteq [N], |A|=k}\ \prod_{l\in A}(-1)^{x[l]+y[l]}\nonumber \\
&= \delta_{jk}{N\choose k}Q_k\big ( \|x-y\|, N,\ \frac{1}{2}\big ),\
j,k\in [N],
\nonumber
\end{align}
noting that  $(-1)^{x[l]+y[l]} = (-1)^{\big (x[l]-y[l] \mod 2 \big )}$ and using (\ref{Qspin:00}).
\end{Remark}
\begin{Example}\label{secondex:00}
In the simple random walk in Section \ref{randomwalkexamples} \ref{item:02} the coefficient of the $k$ spin in (\ref{exchcirc:00a}) is $2^{-\frac{N}{2}}$ times
\begin{equation*}
\frac{1}{\sqrt{
\Big (1 + \frac{\alpha}{1-\alpha}\frac{2k}{N}\Big )
}
}
= \mathbb{E}\big [Y_{\frac{1}{2}}^k\big ],
\label{coeft:99}
\end{equation*}
where
\[
Y_{\frac{1}{2}} = \exp \Big \{-\frac{2\alpha}{N(1-\alpha)} U_{\frac{1}{2}} \Big \},
\]
with $U_{\frac{1}{2}}$ a Gamma$(\frac{1}{2})$ random variable.
Later we find that for de Finetti models in Section \ref{randomwalkexamples} \ref{item:01}, where $\rho_k$  is the $k^\text{\rm th}$ moment of a random variable on $[-1,1]$ in (\ref{deF:101}), there exists a random variable $Y_{\frac{1}{2}}$, depending on the de Finetti measure, such that
\[
 \frac {1}{\sqrt{1 + \frac{\alpha}{1-\alpha}(1-\rho_k)}}=
 \mathbb{E}\big [Y_{\frac{1}{2}}^k\big ].
\]
\end{Example}
A Markov sequence of coupled Gaussian fields 
with increasing dimensions in $N \in \mathbb{N}$ can be defined. 
\begin{Proposition}\label{Markov:000}
Let $x \in {\cal V}_\infty$ and denote $x_N$ as the first $N$ entries in $x$.
Let $Z \in {\cal V}_\infty$ be well defined. Assume that $\alpha$ does not depend on $N$.
Denote ${\cal G}_N = \big (g_{x_N}\big )_{x_N \in {\cal V}_N}$ for $N\in \mathbb{N}$, with
${\cal G}_0 = (g_\emptyset)$. $({\cal G}_N)$ are coupled with the same random variables $(g_A)$ where elements of ${\cal G}_N$ are linear forms in $(g_A)_{A\subseteq [N]}$.
In the construction, for any given $x \in {\cal V}_\infty$,
each Gaussian sequence $(g_{x_N})_{N\in \mathbb{N}}$ is a Markov chain and therefore $\big ({\cal G}_N\big )_{N\in \mathbb{N}}$ is a Markov chain.
Conditional means and covariances are, for $x_N,y_N \in {\cal G}_N$,
\begin{align}
\mathbb{E}\big [g_{{ x}_N}\mid {\cal G}_{N-1}\big ] &= 2^{-1/2}g_{x_{N-1}}\nonumber\\
\text{\rm Cov}\big (g_{{x}_N}, g_{{y}_N}\mid {\cal G}_{N-1}\big )
& = \frac{1}{2^N}\sum_{A\subseteq [N], A \supseteq \{N\}} \frac{1}{1 + \frac{\alpha}{1-\alpha}(1-\rho_A)}
\cdot \prod_{k\in A}(-1)^{{ x}[k]+{y}[k]}
\nonumber
\end{align}
\end{Proposition}
\begin{proof}
From the decomposition (\ref{decomposition:a00}) in Proposition \ref{prop:0},
\begin{align}
g_{{x}_N} &= 2^{-N/2}\Big \{g_\emptyset 
+ \sum_{A\subseteq [N-1], A \ne \emptyset} \frac{1}{\sqrt{1 + \frac{\alpha}{1-\alpha}(1-\rho_A)}}
\cdot \prod_{k\in A}(-1)^{{x}[k]}\cdot g_A\Big \}\nonumber\\
&+ 2^{-N/2}\sum_{A\subseteq [N], A \supseteq \{N\}} \frac{1}{\sqrt{1 + \frac{\alpha}{1-\alpha}(1-\rho_A)}}
\cdot \prod_{k\in A}(-1)^{{x}[k]}\cdot g_A\Big \}\nonumber\\
&= 2^{-1/2}g_{{ x}_{N-1}}+ 2^{-N/2}\sum_{A\subseteq [N], A \supseteq \{N\}} \frac{1}{\sqrt{1 + \frac{\alpha}{1-\alpha}(1-\rho_A)}}
\cdot \prod_{k\in A}(-1)^{{x}[k]}\cdot g_A.
\nonumber
\end{align}
Now ${\cal G}_{N-1}$ is independent of $(g_A)_{A\supseteq \{N\}}$  showing the Markov structure. The conditional means and covariances are straightforward to calculate.
\end{proof}
\begin{Corollary}
A full triangular decomposition of $g_{{x}_N} \in {\cal G}_N$ into $N+1$ independent sums is
\begin{equation*}
g_{{x}_N} = 2^{-N/2}g_\emptyset+
 2^{-N/2} \sum_{j=1}^N\sum_{A_j} \frac{1}{\sqrt{1 + \frac{\alpha}{1-\alpha}(1-\rho_{A_j})}}
\cdot \prod_{k\in A_j}(-1)^{{x}[k]}\cdot g_{A_j},
\label{triangle:000}
\end{equation*} 
where for $j \in [N]$, $A_j$ are subsets of $[N]$ with maximal element $j$.
\end{Corollary}
\subsection{A centered Gaussian field}\label{GaussCentre:sec}
Define a modified field $(g^\circ_x)$  by centering
$
g^\circ_x = g_x - 2^{-N}\sum_{y\in {\cal V}_N}g_y. 
$
Then the baseline component $g_\phi$ is removed because
$
2^{-N}\sum_{y\in {\cal V}_N}g_y = 2^{-N/2}g_\emptyset
$
and
\begin{equation*}
g^\circ_x = \frac{1}{2^{N/2}}\Big \{
\sum_{A\subseteq [N], A \ne \emptyset}
 \frac{1}{\sqrt{1 + \frac{\alpha}{1-\alpha}(1-\rho_A)} }\cdot
 \prod_{k\in A}(-1)^{x[k]}\cdot g_A\Big \}
 \label{decomposition:a000}.
\end{equation*}
$(g_x)$ has positive covariances, however $(g^\circ_x)$ can have positive or negative covariances with
\[
\text{\rm Cov}(g^\circ_x,g^\circ_y) = \text{Cov}(g_x,g_y) - \frac{1}{2^N}.
\]
 Often the covariance function is non-negative, however \cite{DE2002} construct Gaussian Fields where the covariance function can be negative.
 
Independence properties of the field are clearly seen from $(g^\circ_x)$ in Proposition \ref{prop:2}. Recall the notation that will be used of
$x(C_x) = (x[k];k \in C_x \subseteq [N])$.
\begin{Proposition}\label{prop:2}
For $C_x\subseteq [N]$ define a marginal average $g^\circ_{x(C_x)}
$, where the $N- |C_x|$ coordinates of $x$ not in $C_x$ are summed over by
\[
g^\circ_{x(C_x)}
= 2^{(N-|C_x|)/2}
\sum_{x:x[k] \in \{0,1\}; k \notin C_x}g^\circ_{x}.
\]
Then
\[
g^\circ_{x(C_x)}
=  2^{-|C_x|/2}
\sum_{A \subseteq C_x, A \ne \emptyset}\ \prod_{j\in A}(-1)^{{x}[j]}\ \cdot
\frac{1}{\sqrt{1 + \frac{\alpha}{1-\alpha}(1-\rho_A)}}\ \cdot
g_A.
\]
$(g^\circ_{x(C_x)})$ and $(g^\circ_{y(C_y)} )$ are independent if $C_x\cap C_y=\emptyset$ (even if $x=y$).\\
More generally if $C_x\cap C_y\ne\emptyset$.
\begin{equation}
\text{\rm Cov}(g^\circ_{x(C_x)},g^\circ_{y(C_y)}) 
=  2^{-|C_x|/2 - |C_y|/2 }
\sum_{A \subseteq C_x\cap C_y, A \ne \emptyset}\ \prod_{j\in A}(-1)^{x[j]+y[j]}\ \cdot
\frac{1}{1 + \frac{\alpha}{1-\alpha}(1-\rho_A)}.
\label{notempty:00}
\end{equation}
\end{Proposition}
\begin{proof}
The proof follows by noting that summing in $g_x^\circ$ for an entry $j \not\in C_x$
\[
(-1)^{x[j]=0} + (-1)^{x[j]=1} = 1 + -1 = 0.
\]
For the independence property $(g^\circ_{x(C_x)})$ and $(g^\circ_{y(C_y)} )$ sum over disjoint values in $(g_A)$ so they are independent. The general form (\ref{notempty:00}) is straightforward to derive.
\end{proof}
Viewing indexing for fixed $x$ with $C_x$ varying as
\begin{equation}
(g^\circ_{x(C_x)})_{ C_x \subseteq [N], C_x \ne \emptyset}
\label{subsetprocess:00}
\end{equation}
the process (\ref{subsetprocess:00}) has independent increments. $C_x$ itself does not depend on $x$ as it only selects a subset of $[N]$ with $\dim x(C_x) = \dim C_x$, however it is convenient to index 
$(g^\circ_{x(C_x)})$ taking  choices of subsets $C_x$ for a given $x$.
An example is $N=5$, $x=(11010)$, $C_x=\{2,3,4\}$, $x(C_x) = (101)$.
\begin{Corollary}\label{kspin:00}
If $(Z_t)$ are exchangeable then
\begin{equation}
g^\circ_{x(C_x)}
=  2^{-|C_x|/2}
\sum_{k=1}^{|C_x|}
\frac {1}{\sqrt { 1 + \frac{\alpha}{1-\alpha}(1-\rho_k)}}
\cdot S_k\big (x;C_x,(g_A)\big )
\label{exchcirc:00}
\end{equation}
\end{Corollary}
\begin{Remark}
Gaussian spins for $x\in {\cal V}_N$ can be found as linear sums of $g^\circ_{x(C_x)}$ by inversion of the expressions
\begin{equation*}
\sum_{C \subseteq [N], |C|=l}g^\circ_{x(C)}
= 2^{-\frac{l}{2}}\sum_{k=1}^l{N-k\choose l-k}\frac {1}{\sqrt { 1 + \frac{\alpha}{1-\alpha}(1-\rho_k)}}
\ S_k\big (x;[N],(g_A)\big ).
\label{spin:354}
\end{equation*}
\end{Remark}

\subsection{Gaussian fields on ${\cal V}_\infty$}\label{GaussInfty:sec}
$\mathbb{P}(y\mid x) = \delta_{yx}$ for all $x,y \in {\cal V}_N$
is equivalent to $\rho_A^{(N)}=1$ for all $A \subseteq [N]$. In a process where $\alpha, \rho_A$ depend on $N$,  $\alpha_N \to 1$ and the transition function behaves like a delta function with $\rho_A^{(N)} \to 1$ for all $A\subseteq [N]$ suppose
\begin{equation*}
b_A := \lim_{N\to \infty} \frac{\alpha_N} {1-\alpha_N}(1-\rho^{(N)}_A)
\label{rholim:00}
\end{equation*}
is a proper finite limit for all finite subsets $A \subseteq \mathbb{Z}_+$. Then for $x\in {\cal V}_\infty$ and $|C_x| < \infty$ set
\begin{equation}
g^\infty_{x(C_x)} :=
 2^{-|C_x|/2}
\sum_{A \subseteq C_x, A \ne \emptyset}\prod_{j\in A}(-1)^{{x}[j]}\ \cdot
\frac{1}{\sqrt{1 + b_A}}\ \cdot
g_A,
\label{increments:001}
\end{equation}
which is well defined. In an exchangeable limit
\begin{equation*}
g^\infty_{x(C_x)}
=  2^{-|C_x|/2}
\sum_{k=1}^{|C_x|}
\frac {1}{\sqrt { 1 + b_k}}
\cdot S_k\big (x;C_x,(g_A)\big ).
\label{exchcirc:00aa}
\end{equation*}
\begin{Proposition}
A Gaussian field $\big (g^\infty_{x(C_x)}\big )$ for $x \in {\cal V}_\infty$, $C_x$ a finite subset of 
$\mathbb{Z}_+$ is well defined by (\ref{increments:001}). There is weak convergence of 
$\big (g^{(N)}_{x(C_x)}\big )$ to $\big (g^\infty_{x(C_x)}\big )$.
The process has independent increments such that $g^\infty_{x(C_x)}$ and $g^\infty_{y(C_y)}$ are independent if $C_x\cap C_y = \emptyset$.
\end{Proposition}
\begin{proof}
The \iid  collection $(g_A)$ of N$(0,1)$ random variables is well defined for finite subsets $A$ of $\mathbb{Z}_+$. The number of subsets in the indexing set is countable. $\big (g^\infty_{x(C_x)}\big )$ is then well defined by finite sums 
(\ref{increments:001}) of random variables in $(\xi_A)$. Convergence of any finite collection
$g^{(N)}_{x_1(C_{x_1})},\ldots, g^{(N)}_{x_r(C_{x_r})}$ to 
$g^{\infty}_{x_1(C_{x_1})},\ldots, g^{\infty}_{x_r(C_{x_r})}$ follows because the variables are normal and covariances converge. Weak convergence here is in the sense of finite dimensional distributions converging. For Gaussian variables with zero means it is sufficient that their covariances converge.
 The independent increments property follows because $g^\infty_{x(C_x)}$ and $g^\infty_{y(C_y)}$ are sums over disjoint variables $g_A$.
\end{proof}
\subsection{Gaussian fields from a de Finetti random walk}\label{deF:sec}
In a random walk (\ref{rw:000}) on ${\cal V}_\infty$ when $(Z_t)$ are \iid  de Finetti sequences then, for each $t$, $X_t$ is a de Finetti sequence. The Gaussian field is determined by how these measures change over time. If $N$ entries are selected from $(X_t)$ then the random walk $(X_t^{(N)})$ on ${\cal V}_N$ is coupled with $(X_t)$. There is a distinction however from a random walk in ${\cal V}_N$ constructed by taking $N$ entries for $(Z_t^{(N)})$ from a de Finetti sequence with killing probability $1-\alpha_N$ depending on $N$.

The concept of a probability generating functional is needed in the paper.
A finite point process $M(\cdot)$ with points $(\xi_j)$ in a state space ${\cal X}$  has a probability generating functional
\begin{align*}
G[f] &= \mathbb{E}\big [\exp\Big (\int_{\cal X}\log f(x) M(dx) \Big )\Big ]\\
&= \mathbb{E}\big [\prod_{j=1}^\infty f(\xi_j)\big ]
\end{align*}
for a suitable class of functions $0 \leq f \leq 1$ where $1-f$ vanishes outside some bounded set. Suppose $f \in {\cal V(X)}$, the class of real valued Borel functions on a complete separable metric space. In this paper ${\cal X}$ is $[0,1]$ or $[-1,1]$.
A Poisson point process where the points have a state space ${\cal X}$ has a \pgfl 
\[
G[f]= \exp \big \{- \int_{\cal X} \big (1-f(\xi)\big )\mu(d\xi)\Big \},
\]
where the intensity measure $\mu$ is a non-negative measure on ${\cal X}$. $f$ has to be a function for which the integral is finite.  For an introduction to probability generating functionals see \cite{W1972, DVJ2008}. 

The next proposition identifies how the de Finetti measures change over time $t$. It is sometimes convenient to think of spin sequences $\pm 1$. If $\nu$ is a de Finetti measure on $[0,1]$ define $\nu^\spin$ on $[-1,1]$ by a change of variable,
 where for a measurable set $B \subseteq [-1,1]$
 \begin{equation*}
 \int_B\nu^\spin(d\xi) = \int_{1-2\omega \in B}\nu(d\omega).
 \label{shift:123}
 \end{equation*}
\begin{Proposition}\label{propdefmeasures:00}
In the random walk (\ref{rw:000}) let $(Z_t)_{t \in \mathbb{Z}_+}$ be \iid  de Finetti sequences with measure $\nu$.
$(X_t)$ are de Finetti sequences with measures $\varphi_t$ such that for a measurable set $A \subseteq [-1,1]$
\begin{align}
\int_A\varphi^\spin_{t+1}(d\psi_{t+1}) 
= \int_{\xi\psi_t \in A}\nu^\spin(d\xi)\varphi_t^\spin(d\psi_t),\ t \in \mathbb{N},
\label{deF:250}
\end{align}
with a recursive solution
\begin{align}
\int_A\varphi^\spin_{t}(d\psi_{t}) 
= \int_{\prod_{k=0}^t \psi_k \in A}\otimes_{k=1}^t\nu^\spin(d\psi_k)
\otimes\varphi^\spin_0(d\psi_0),
\label{deF:251}
\end{align}
where the initial spin measure is $\varphi^\spin_0$.
The de Finetti measure after the $t^{\text{th}}$ transition is $\varphi_t$, specified by 
\[
\int_A\varphi_t(d\omega) = \int_{\frac{1-\psi}{2} \in A}\varphi_t^\spin(d\psi)
\]
for measurable sets $A \subseteq [0,1]$.
\end{Proposition}
\begin{proof}
An entry $X_{t+1}[k] = 1$ if $X_t[k]=1, Z_t[k]=0$ or $X_t[k]=0, Z_t[k]=1$,
the same as $X_{t+1}[k] = X_t[k](1-Z_t[k]) + (1- X_t[k])Z_t[k]$ or $1-2X_{t+1}[k]=(1-2Z_t[k])(1-2X_t[k])$.
Therefore for $B\subseteq [0,1]$
\begin{equation*}
\int_B\varphi_{t+1}(d\phi_{t+1}) 
= \int_{\phi_t(1-\omega) + (1-\phi_t)\omega \in B}\nu(d\omega)\varphi_t(d\phi_t),
\label{spinback:00}
\end{equation*}
which is equivalent after a change of variable to (\ref{deF:250}). In the change of variable
$\omega = (1-\xi)/2$, $\phi_t = (1 - \psi_t)/2$ and 
$\phi_t(1-\omega) + (1-\phi_t)\omega = (1-\xi\psi_t)/2$.
A recursive solution is clearly (\ref{deF:251}). 
\end{proof}
Expressed in terms of random variables Proposition \ref{propdefmeasures:00} is equivalent to the following.
Let $\xi_1,\ldots, \xi_t$ be \iid random variables with a probability measure $\nu^\spin$ and $\xi_0$ a random variable with probability measure $\varphi_0^\spin$. Then the de Finetti measure of $X_t$, $\varphi_t$, is the probability measure of
$\frac{1}{2}\big ( 1 - \xi_0\prod_{j=1}^t\xi_j\big )$.
In Corollary \ref{pgflcorr} the \pgfl of a point process
$(\xi_j)_{j=0}^{T_\alpha}$ is calculated. 
\begin{Corollary}\label{pgflcorr}
A point process consisting of $T_\alpha$ \iid points from $\nu^\spin$ and an initial measure $\varphi_0^\spin = \delta_1$ has a \pgfl
\begin{equation}
G[f] 
=  \frac{1}{1+\frac{\alpha}{1-\alpha}\int_{[-1,1]}(1-f(\xi))\nu^\spin(d\xi)}f(1).
\label{pgfl:2600a}
\end{equation}
\end{Corollary}
\begin{proof}
The \pgfl of points $(\xi_j)_{j=0}^{t}$ from an initial measure $\delta_1$ and then from $\nu^\spin$ is equal to
\[
G_t[f] = \Big (\int_\spin f(\xi)\nu^\spin(d\xi)\Big )^tf(1).
\]
Therefore the \pgfl of $T_\alpha$ points  is $(1-\alpha)\sum_{t=0}^\infty\alpha^t G_t[f]$ which is equal to (\ref{pgfl:2600a}).
\end{proof}
The construction of the \pgfl  (\ref{pgfl:2600a}) is through a finite number of $T_\alpha$ points. However in the unconditional point process with \pgfl (\ref{pgfl:2600a}) there are points $(\xi_j)_{j=1}^\infty$.
\begin{Proposition}\label{propde:90}
Let $(Z_t)$ be \iid  de Finetti sequences with measures $\nu$. $X_0=0$ and $X_{T_\alpha}$ is the random walk (\ref{rw:000}) on ${\cal V}_\infty$ with stopping time according to (\ref{killing:00}). Then with $\rho_k=\mathbb{E}\big [(-1)^{Z[1]+\cdots + Z[k]}\big ]$,
\begin{equation}
\frac{1}{1 + \frac{\alpha}{1-\alpha}(1-\rho_k)} = \mathbb{E}\big [Y^k\big ],\ k \in \mathbb{N},
\label{YYY:000}
\end{equation}
where $Y$ is a random variable on $[-1,1]$ which is the product of points in a point process with \pgfl (\ref{pgfl:2600a}).
%
Take $x_N,y_N \in {\cal V}_N$ as $N$ particular entries in the ${\cal V}_\infty$ random walk. The Green function for these $N$ entries has an expression
\begin{equation}
(1-\alpha)G(x_N,y_N;\alpha)= \mathbb{E}\Big [ 
\Big (\frac{1}{2}-\frac{Y}{2}\Big )^{\|y_N-x_N\|}\Big (\frac{1}{2}+\frac{Y}{2}\Big )^{N-\|y_N-x_N\|}\Big ].
\label{GreendeF:00}
\end{equation}
\end{Proposition}
\begin{proof}
Let $T_\alpha$ be a geometric random variable with distribution (\ref{killing:00}).
Consider a point process $\{\xi_j\}$, with elements in $[-1,1]$,  constructed from $T_\alpha$ independent points with probability measure $\nu^\spin$ and an initial point process with probability measure $\delta_{\{1\}}$. The initial spin measure,  degenerate at $\{1\}$, arises from $X_0=0$. The \pgfl  of this constructed point process is (\ref{pgfl:2600a}).
By definition $Y= \prod_{j=1}^\infty\xi_j$. Set $f_k(\xi) = \xi^k$. Then
\begin{align*}
\mathbb{E}\big [Y^k\big ] &= \mathbb{E}\Big [\Big (\prod_{j=1}^\infty\xi_j\Big )^k\Big ]\nonumber\\
&= \mathbb{E}\big [\prod_{j=1}^\infty f_k(\xi_j)\big ]\nonumber\\
&=  \frac{1}{1+\frac{\alpha}{1-\alpha}\int_{[-1,1]}(1-f_k(\xi))\nu^{[-1,1]}(d\xi)}\nonumber\\
&=  \frac{1}{1+\frac{\alpha}{1-\alpha}\int_{[-1,1]}(1-\xi^k)\nu^{[-1,1]}(d\xi)},
\end{align*}
implying (\ref{YYY:000})
since $\rho_k = \int_\spin \xi^k\nu^\spin(d\xi)$.
Then (\ref{GreendeF:00}) follows from (\ref{introeigen:00}). Even though $Y \in [-1,1]$, $\mathbb{E}\big [Y^k\big ] \geq 0$.
\end{proof}
\begin{Corollary}\label{corrkill:00}
Let $(X_t)$, with $X_0=x$, be a random walk (\ref{rw:000}) on ${\cal V}_\infty$ with a de Finetti measure $\nu$ for $Z$. Let $T_\alpha$ be a stopping time with distribution (\ref{killing:00}). 
Then $X_{T_\alpha} - x \mod 2$ is a de Finetti sequence with a measure which is the probability measure of $V := \frac{1}{2}- \frac{Y}{2}$.
\end{Corollary}
\begin{proof} The probability distribution of any $N \in \mathbb{Z}_+$ ordered entries from $X_{T_\alpha}-x$ with $n$ entries one, from (\ref{GreendeF:00}), is 
\[
\mathbb{E}\big [V^n(1-V)^{N-n}\big ].
\]
This implies the statement of the corollary.
\end{proof}
\begin{Remark}\label{Y:000}
Proposition \ref{propde:90} is written in terms of a de Finetti model. However all that it requires for (\ref{GreendeF:00}) to be true is that if $\rho_k^{(N)}$ is an eigenvalue in the spectral expansion of the transition function of $(X_t)$ which has exchangeable entries then
there exists a random variable $Y$, which may depend on $N$, such that
\[
\mathbb{E}\big [Y^k\big ] 
= \frac{1}
{1 + \frac{\alpha}{1-\alpha}(1-\rho_k^{(N)})},\ k\in [N].
\]
In the simple random walk of Example \ref{secondex:00} 
\[
Y = \exp\Big \{-\frac{2\alpha}{N(1-\alpha)}U\Big \},
\]
where $U$ is exponential with mean 1. With a change of variable the density of $Y$ is
\[
{\scriptstyle\frac{N(1-\alpha)}{2\alpha}}y^{\frac{N(1-\alpha)}{2\alpha}-1},\ 0 < y < 1,
\]
in agreement with the calculation in Equation (5) of \cite{CG2025}.
\end{Remark}
\begin{Remark}\label{remarkmod:00}
We would like to understand a product of points taken from $\otimes_{k=1}^t\nu^\spin(d\xi_k)$.  Let $(\xi_k)$ be \iid  random variables with probability measure $\nu^\spin$. Suppose that $\nu^\spin$ does not have an atom at zero. The joint probability that $\prod_{k=1}^t\xi_k$ is non-negative and the Laplace transform of 
$-\log \prod_{k=1}^t|\xi_k|$ with parameter $\theta$ is
\begin{align}
&\mathbb{E}\Big [ \mathbb{I}\big \{\prod_{j=1}^t\xi_j \geq 0\big \}e^{-\theta \big (-\log \prod_{k=1}^t|\xi_k|\big )}\Big ]
\nonumber\\
&=\sum_{0\leq k \leq t: k\text{~even}}
{t\choose k}
\Big (\int_{[-1,0]}|\xi|^\theta\nu^\spin(d\xi)\Big )^k\Big (\int_{(0,1]}|\xi|^\theta\nu^\spin(d\xi) \Big )^{t-k}
\nonumber
\end{align}
which is the sum of coefficients of even powers of $s$ in
\begin{equation}
\Bigg (s\int_{[-1,0]}|\xi|^\theta\nu^\spin(d\xi)+ \int_{(0,1]}|\xi|^\theta\nu^\spin(d\xi)\Bigg )^t.
\label{modxx:00}
\end{equation}
 The marginal Laplace transform of $-\log \prod_{k=1}^t|\xi_k|$ is
\begin{equation}
\mathbb{E}\Big [e^{-\theta \big (-\log \prod_{k=1}^t|\xi_k|\big )}\Big ]
= \Bigg (\int_\spin|\xi|^\theta\nu^\spin(d\xi)\Bigg )^t.
\label{LT:74}
\end{equation}
If $\nu^\spin$ has a symmetric distribution around zero, then the sign of 
$\prod_{j=1}^t\xi_j$ is positive or negative with probability $\frac{1}{2}$, independently of 
 $-\log \prod_{k=1}^t|\xi_k|$.  With this symmetry
 \begin{align}
 \int_\spin(1-\xi^k)\nu^\spin(d\xi)
 = \begin{cases} 
 1&k\text{~odd}\\
 2\int_{(0,1]}(1-\xi^k)\nu^\spin(d\xi)&k \text{~even}
 \end{cases}.
 \label{symmetric:01}
 \end{align}
\end{Remark}
\begin{Proposition}\label{propkilled:00}
In the de Finetti random walk on ${\cal V}_\infty$ in Proposition \ref{propdefmeasures:00}
suppose there is a stopping time $T_\alpha$ with distribution (\ref{killing:00}). 
Suppose that $\nu^\spin$ and the initial measure $\varphi_0^\spin$ do not have atoms at zero.
Then the Laplace transform of $-\log \prod_{j=0}^{T_\alpha}|\xi_j|$ 
and probability of a positive/negative
sign of 
$\prod_{k=0}^{T_\alpha}\xi_k$
is 
\begin{align}
&\frac{1}{2}\Bigg ( 
1 + \frac{\alpha}{1-\alpha} \Big (\int_{[-1,1]}(1-|\xi|^\theta)\nu^\spin(d\xi) \Big )\Bigg )^{-1}
\cdot\Bigg (\int_{[-1,1]}|\xi|^\theta\varphi^\spin_0(d\xi)\Bigg )\nonumber\\
&\pm 
\frac{1}{2}(1-\alpha)\Bigg ( 
1 - \alpha \Big (-\int_{[-1,0]}|\xi|^\theta\nu^\spin(d\xi)  +\int_{(0,1]}|\xi|^\theta \nu^\spin(d\xi)\Big )\Bigg )^{-1}\nonumber\\
&~~~~~~\times\Bigg (-\int_{[-1,0]}|\xi|^\theta\varphi^\spin_0(d\xi)+\int_{(0,1]}|\xi|^\theta\varphi^\spin_0(d\xi)\Bigg ).
\label{LT:45}
\end{align}
The marginal Laplace transform of $-\log \prod_{j=0}^{T_\alpha}|\xi_j|$ is
\begin{equation}
\Bigg ( 1 + \frac{\alpha}{1-\alpha}\int_\spin\big (1-|\xi|^\theta\big )\nu^\spin(d\xi)\Bigg )^{-1}\Big (\int_{[-1,1]}|\xi|^\theta\varphi_0^\spin(d\xi)\Big )
\label{LT:46}
\end{equation}
and the marginal probability of a positive/negative
sign of 
$\prod_{k=0}^{T_\alpha}\xi_k$ is 
\begin{align}
\frac{1}{2}\Bigg (
1 \pm \frac{1}{1 + \frac{2\alpha}{1-\alpha}\nu\big( [\frac{1}{2},1]\big )}
\Big (1 -2 \varphi_0\big( [\frac{1}{2},1]\big )\Big )\Bigg ).
\label{spinprob:00}
\end{align}
\end{Proposition}
\begin{proof}
From (\ref{modxx:00}) we require the sum of coefficients of even/odd powers of $s$ in
\begin{align*}
&H(s) := \sum_{t=0}^\infty(1-\alpha)\alpha^t
\Bigg (s\int_{[-1,0]}|\xi|^\theta\nu^\spin(d\xi)+ \int_{(0,1]}|\xi|^\theta\nu^\spin(d\xi)\Bigg )^t\nonumber\\
&~~~~~~\times\Bigg ({s}\int_{[-1,0]}|\xi|^\theta\varphi^\spin_0(d\xi)+ \int_{(0,1]}|\xi|^\theta\varphi^\spin_0(d\xi)\Bigg )\nonumber\\
&= (1-\alpha)\Bigg ( 
1 - \alpha \Big (s\int_{[-1,0]}|\xi|^\theta\nu^\spin(d\xi)  +\int_{(0,1]}|\xi|^\theta \nu^\spin(d\xi)\Big )\Bigg )^{-1}\nonumber\\
&~~~~~~\times\Bigg (s\int_{[-1,0]}|\xi|^\theta\varphi^\spin_0(d\xi)+ \int_{(0,1]}|\xi|^\theta\varphi^\spin_0(d\xi)\Bigg )
\end{align*}
 which is equal to (\ref{LT:45}). 
 In the calculation the even coefficient sum is $\frac{1}{2}\big (H(1)+H(-1)\big )$ and the odd coefficient sum is $\frac{1}{2}\big (H(1)-H(-1)\big )$, which simplify to (\ref{LT:45}) knowing
 \begin{align*}
H(1)&=\Bigg ( 
1 + \frac{\alpha}{1-\alpha} \Big (\int_{[-1,1]}(1-|\xi|^\theta)\nu^\spin(d\xi) \Big )\Bigg )^{-1}
\cdot\Bigg (\int_{[-1,1]}|\xi|^\theta\varphi^\spin_0(d\xi)\Bigg ),\\
H(-1) &= (1-\alpha)\Bigg ( 
1 - \alpha \Big (-\int_{[-1,0]}|\xi|^\theta\nu^\spin(d\xi)  +\int_{(0,1]}|\xi|^\theta \nu^\spin(d\xi)\Big )\Bigg )^{-1}\nonumber\\
&~~~~~~\times\Bigg (-\int_{[-1,0]}|\xi|^\theta\varphi^\spin_0(d\xi)+\int_{(0,1]}|\xi|^\theta\varphi^\spin_0(d\xi)\Bigg ).
\end{align*}
 The marginal probability of a positive/negative
sign of 
$\prod_{k=0}^{T_\alpha}\xi_k$
is 
\begin{align*}
& \frac{1}{2}\Bigg [ 1 \pm
(1-\alpha)\Bigg ( 
1 - \alpha \Big (-\nu^\spin([-1,0])  + \nu^\spin((0,1])\Big )\Bigg )^{-1}\Bigg ]\nonumber\\
&~~~~~~\times\Bigg (-\varphi^\spin_0([-1,0])+\varphi^\spin_0((0,1])\Bigg )\\
& =\frac{1}{2}\Bigg [ 1 \pm
(1-\alpha)\Bigg ( 
1 - \alpha  + 2\alpha\nu^\spin([-1,0])\Bigg )^{-1}\Bigg ]\\
&~~~~~~\times\Bigg (1 -2 \varphi^\spin_0([-1,0])\Bigg )
\end{align*}
which simplifies to (\ref{spinprob:00}), recalling that $[-1,0]$ in the spin measures on $\spin$ corresponds to $[\frac{1}{2},1]$ in the measures on $[0,1]$.
\end{proof}
\begin{Example}\label{Betaexample:00}
A particular example is when $\nu^\spin$ is symmetric around zero and $2\nu^\spin$ is a Beta$(a,b)$ measure on $[0,1]$ with $a,b>0$. The Laplace transform (\ref{LT:74}) is then equal to
\begin{align}
\mathbb{E}\Big [e^{-\theta \big (-\log \prod_{k=1}^t|\xi_k|\big )}\Big ]
&= \Bigg (\frac{\Gamma(a+\theta)\Gamma(a+b)}{\Gamma(a+b+\theta)\Gamma(a)}\Bigg )^t\label{Betaex:01}\\
&= \Big ( 1 + \frac{\theta}{a+b-1}\big )^{-t} \cdots \Big (1 + \frac{\theta}{a} \Big )^{-t}.
\label{Betaex:00}
\end{align}
Equation (\ref{Betaex:01}) holds generally, and (\ref{Betaex:00}) when $b\geq 1$ is an integer.
The transform of (\ref{Betaex:00}) is of a convolution of Gamma random variables.
Let $U_1,\ldots U_b$ be \iid  Gamma$(t)$ random variables. Then
\begin{equation*}
\prod_{j=1}^t|\xi_j| =^{\cal D}e^{- \sum_{j=0}^{b-1}\frac{1}{a+j}U_j}
\label{prodgamma:02}
\end{equation*}
and $\prod_{j=1}^t\xi_j$ has the same distribution with a random $\pm 1$ sign.
If $b=1$, so
\[
\nu^\spin(d\xi) = \frac{1}{2}a|\xi|^{a-1}d\xi, \ \xi\in (-1,1),
\]
then there is only one term in (\ref{Betaex:00}) and the Laplace transform corresponds 
to $a^{-1}U_1$. The density of $\zeta:=\prod_{j=1}^{t}\xi_j$ is then 
\begin{equation*}
g_t(\zeta) = \frac{a}{2\Gamma (t)}|\zeta|^{a-1}\Big (-a \log |\zeta|\Big )^{t-1},\ \zeta \in (-1,1).
\label{prodgamma:00}
\end{equation*}
If $X_0=0$ then $\zeta_\alpha := \prod_{j=1}^{T_\alpha}\xi_j$ has an atom of $1-\alpha$ at unity and 
continuous density of
\begin{equation*}
\alpha\frac{a}{2\Gamma (t)}|\zeta_\alpha|^{a-1}\frac{(1-\alpha)}{1 - \alpha(-a\log|\zeta_\alpha|)},
\ \zeta_\alpha \in (-1,1).
\label{prodgamma:03}
\end{equation*}
$Y = \frac{1}{2}\Big (1 - \prod_{j=1}^{T_\alpha}\xi_j\Big )$ has an atom of $1-\alpha$ at zero and continuous density of
\begin{equation*}
\alpha\frac{a}{4\Gamma (t)}|1-2y|^{a-1}\frac{(1-\alpha)}{1 - \alpha(-a\log|1-2y|)},
\ y \in (0,1).
\label{prodgamma:04}
\end{equation*}
A general representation in terms of Gamma random variables extending the case to when $b$ is not an integer can be found. Euler's infinite product formula 6.1.3 in \cite{AS1972}, for example, is
\begin{equation} 
\frac{1}{\Gamma(z)} = ze^{\gamma z}\prod_{k=1}^\infty\big ( 1 + \frac{z}{k}\big )e^{-\frac{z}{k}},
\label{Euler:00}
\end{equation}
where $\gamma$ is Euler's constant. From (\ref{Euler:00})
\begin{equation*}
\frac{\Gamma(a+b)}{\Gamma(a+b+\theta)}
= \big ( 1 + \frac{\theta}{a+b}\big ) e^{-\gamma \theta}
\prod_{k=1}^\infty\big ( 1 + \frac{\theta}{a+b+k}\big )e^{-\frac{\theta}{k}},
\end{equation*}
with a similar product for $\frac{\Gamma(a)}{\Gamma(a+\theta)}$. 
It follows that for $t\in \mathbb{Z}_+$,
\begin{align}
&\Bigg (\frac{\Gamma(a+\theta)\Gamma(a+b)}{\Gamma(a+b+\theta)\Gamma(a)}\Bigg )^t\nonumber \\
&=\prod_{k=0}^\infty \Bigg (\frac{1 + \frac{\theta}{a+b+k}}{1 + \frac{\theta}{a+k}}\Bigg )^t\label{Euler:02}\\
&= \prod_{k=0}^\infty \Big (\frac{a+k}{a+b+k}\Big )^t
\cdot \Big ( 1 - \frac{b}{a+b+k}\Big (1 + \frac{\theta}{a+b+k}\Big )^{-1}\Big )^{-t}.
\label{Euler:01}
\end{align}
Note that the product (\ref{Euler:02}) converges because
\begin{equation}
\log \Big (1 + \frac{\theta}{a+b+k}\Big )
-\log \Big (1 + \frac{\theta}{a+k}\Big ) 
= \frac{-b}{(a+b+k)(a+k)} + R_k,
\label{series:66}
\end{equation}
where 
\begin{align*}
R_k &\leq \Big |\log \Big (1 + \frac{\theta}{a+b+k}\Big ) - \frac{\theta}{a+b+k}\Big |
+ \Big |\log \Big (1 + \frac{\theta}{a+k}\Big ) - \frac{\theta}{a+k}\Big |\\
&\leq \frac{\theta^2}{2}\Big (\frac{1}{(a+b+k)^2} + \frac{1}{(a+k)^2}\Big ).
\end{align*}
Therefore a series with terms (\ref{series:66}) converges and the product (\ref{Euler:02}) converges.
Let $\big (W_{kj}\big )_{k,j\in \mathbb{N}}$ be \iid random variables such that
$W_{kj}$ is Gamma$\Big (\frac{1}{a+b+k},j\Big )$ and $\big (M_{tk}\big )_{k,\in \mathbb{N}}$ be \iid negative binomial random variables with \pgf\emph{s}
\[
\mathbb{E}\Big [s^{M_{tk}}\Big ] =
\Big (\frac{a+k}{a+b+k}\Big )^t\Big (1 - \frac{b}{a+b+k}s\Big )^{-t},\ k \in \mathbb{N}.
\] 
Then (\ref{Euler:01}) is the Laplace transform of $\sum_{k=0}^\infty W_{kM_{tk}}$ and
a representation is that
\begin{equation*}
\prod_{j=1}^t|\xi_j| = e^{-\sum_{k=0}^\infty W_{kM_{tk}}}.
\label{complicated:77}
\end{equation*}
\end{Example}
\begin{Remark}
Let $X_0=0$ corresponding to $\varphi_0^\spin = \delta_{\{1\}}$ ($\varphi_0 = \delta_{\{0\}}$). If $\alpha \to 1$ and 
$\nu^\spin \to \delta_{\{0\}}$ such that 
\begin{align*}
\frac{\alpha}{1-\alpha}\int_\spin\big (1-|\xi|^\theta\big )\nu^\spin(d\xi)
\to \int_\spin\big (1-|\xi|^\theta\big )\nu_*^\spin(d\xi) < \infty
\end{align*}
for a non-negative measure $\nu_*^\spin$, and the limit tends to zero as $\theta \downarrow 0$, then the Laplace transform (\ref{LT:46}) converges to 
a proper Laplace transform 
\begin{equation*}
\Bigg ( 1 + \int_\spin\big (1-|\xi|^\theta\big )\nu_*^\spin(d\xi)\Bigg )^{-1}.
\label{LT:46a}
\end{equation*}
\end{Remark}
\begin{Remark}\label{Zrandomwalk:00}
The geometric random variable $T_{\alpha}$ with \pgf $(1-\alpha)(1-\alpha s)^{-1}$ is infinitely divisible. Let $T_{\alpha,\phi}$ be a negative binomial random variable with \pgf
$(1-\alpha)^\phi(1-\alpha s)^{-\phi}$, $\phi > 0$. Then
\begin{equation*}
\mathbb{P}\big (T_{\alpha,\phi} = k\big ) = (1-\alpha)^\phi\alpha^k\frac{\phi_{(k)}}{k!},\ k\in \mathbb{N},
\label{Tid:000}
\end{equation*}
where $\phi_{(k)} = \phi(\phi+1)\cdots (\phi+k-1)$. $T_{\alpha,\frac{1}{2}}$ is important in this paper. 

A nice connection of $T_{\alpha,\frac{1}{2}}$ with a simple random walk on $\mathbb{Z}$ is the following.
Let
$S_0=0$, $S_k=V_1+\cdots +V_k$ where $(V_k)$ are \iid random variables with
$\mathbb{P}\big (V=1) = p$, $\mathbb{P}\big (V=-1) = q=1-p$. Choose $4pq = \alpha$, where $\alpha \in (0,1)$ with $p>q$. A calculation gives that $p=\frac{1}{2}\big ( 1 + \sqrt{1-\alpha})$.
Write
\begin{align*}
\mathbb{P}\big (T_{\alpha,\frac{1}{2}}= k\big )
=  (1-\alpha)^{\frac{1}{2}}
\frac{\Big (\frac{1}{2}\Big )_{(k)}}{k!}\alpha^k
= (p-q){2k\choose k}(pq)^k
= (p-q)\mathbb{P}\big (S_{2k}=0\big ).
\end{align*}
The random walk $(S_k)$ is transient because $p > q$, with the probability of a return to the origin of $1 - (p-q)$. 
Eventually the random walk will drift to positive infinity. $\big \{T_{\alpha,\frac{1}{2}}=k\big \}$ is identified as equivalent to a last return to the origin in the random walk at $2k$. That is
\begin{align*}
\mathbb{P}\big (T_{\alpha,\frac{1}{2}}= k\big )
 &= \mathbb{P}\big (S_{2k}=0, S_j \ne 0,\ j > 2k \big ) \nonumber \\
 &=  \mathbb{P}\big (S_{2k}=0, S_j > 0,\ j > 2k \big ).
\end{align*}
Random walk concepts are well explained in \cite{F1991} XI.3 and XIV.4.
\end{Remark}
\begin{Remark}\label{remid:00}
The \pgfl (\ref{pgfl:2600a}) is infinitely divisible in the sense that
\begin{equation}
G_\phi[f]=
\frac{1}{\Big (1+\frac{\alpha}{1-\alpha}\int_{[-1,1]}(1-f(\xi))\nu^\spin(d\xi)\Big )^\phi}
\label{infdiv:00}
\end{equation}
is a \pgfl for all $\phi > 0$. The proof is very similar to that in Proposition \ref{propde:90} when $T_\alpha$ is replaced by $T_{\alpha,\phi}$. 
Alternatively it is seen as a mixed Poisson process \pgfl
\begin{equation*}
\int_0^\infty
\frac{1}{\Gamma(\phi)}\lambda^{\phi - 1}
e^{-\lambda}e^{-\lambda \frac{\alpha}{1-\alpha}\int_{[-1,1]}(1-f(\xi))\nu^\spin(d\xi)}
d\lambda\ .
\label{Poissonphi:00}
\end{equation*}
The \pgfl (\ref{infdiv:00}) arises from a negative binomial point process, originally from  \cite{G1984}. More details about the process are in Example 6.4(b) in \citet{DVJ2003}.
Let $Y_\phi$ be the product of the points $(\xi_{\phi,j})$ in a point process with \pgfl
(\ref{infdiv:00}), then 
\begin{align}
\mathbb{E}\big [Y_\phi^k] &= 
\frac{1}{\Big (1+\frac{\alpha}{1-\alpha}\int_{[-1,1]}(1-\xi^k)\nu^\spin(d\xi)\Big )^\phi}.
\nonumber
\end{align}
A surprising property is that
\begin{equation*}
\mathbb{E}\big [Y^k_\phi\big ] = \mathbb{E}\big [Y^k\big ]^\phi,\ k \in \mathbb{N}.
\label{idproperty:00}
\end{equation*}
If $\nu^\spin$ is symmetric about zero, then from (\ref{symmetric:01}), 
$\mathbb{E}\big [Y^k_\phi\big ]=1$ if $k$ is odd, and if $k$ is even
\begin{equation*}
\mathbb{E}\big [Y^k_\phi\big ]=
\frac{1}{
\Big (1+\frac{\alpha}{1-\alpha}
\big (\nu^\spin(\{0\}\big ) 
+ 2\int_{(0,1]}(1-\xi^k)\nu^\spin(d\xi)\big )\Big )^\phi
}.
\label{symmetry:002}
\end{equation*}
\end{Remark}
\begin{Proposition}\label{deF:340}
Let $(g_x)$ be a Gaussian field on ${\cal V}_\infty$ with covariance function (\ref{GreendeF:00}) in Proposition \ref{propde:90}. The elements of $Z$ in the random walk (\ref{rw:000}) are $N$ random variables from a de Finetti sequence with measure $\nu$ and $X_0$ is taken to have all entries zero.
Suppose that $\nu$ does not have an atom at unity. Then there is a random variable 
$Y_{ \frac{1}{2} }\in (-1,1]$ which is the product of points in a point process with \pgfl
\[
G_{\frac{1}{2}}[f] = 
\frac{1}{\sqrt{\Big (1+\frac{\alpha}{1-\alpha}\int_{[-1,1]}\big (1-f(\xi)\big)\nu^\spin(d\xi)\Big )}}
\]
such that
\begin{equation}
g_x
=  2^{-\frac{N}{2}}\Big \{g_\emptyset +
\sum_{k=1}^{N}
 \mathbb{E}\big [Y_{ \frac{1}{2} }^k\big ]
 \cdot S_k\big (x;[N],(\xi_A)\big )\Big \}\ .
\label{exchcirc:00ab}
\end{equation}
Let 
\begin{equation*}
p_k^{(N)} = {N\choose k}
\frac{\mathbb{E}\big [Y_{ \frac{1}{2} }^k\big ]}
{\mathbb{E}\big [\big ( 1 + Y_{ \frac{1}{2} }\big )^N\big ]},\ k=0,1,\ldots, N.
\label{pkN:00}
\end{equation*}
A probabilistic construction of $(g_x)$ is the following:
\begin{itemize}
\item[(a)]Choose a random subset of size $K=k$ from $[N]$ with probability $p_k^{(N)}$;
\item[(b)] For all $x\in {\cal V}_N$ return  coupled $K$-spin Gaussian variables
\begin{equation}
 g_x = 2^{-\frac{N}{2}}\mathbb{E}\big [\big ( 1 + Y_{ \frac{1}{2} }\big )^N\big ]\times
{N\choose k}^{-1}\cdot S_k\big (x;[N],(\xi_A)\big )\ .
\label{randomKspin:00}
\end{equation}
\end{itemize}
\end{Proposition}
\begin{proof}
Equation (\ref{randomKspin:00}) follows from (\ref{exchcirc:00}), (\ref{exchcirc:00a}) in Proposition \ref{prop:0}
and the construction
\[
\mathbb{E}\big [Y_{ \frac{1}{2} }^k\big ] = 
 \frac {1}{\sqrt{1 + \frac{\alpha}{1-\alpha}(1-\rho_k)}}
\]
in Remark \ref{remid:00}. If $\nu$ does not have an atom at unity, than $\nu^\spin$ does not have an atom at $-1$, making $\mathbb{E}\big [\big ( 1 + Y_{ \frac{1}{2} }\big )^N\big ]>0$.
The probabilistic construction of $(g_x)$ is clear and interesting. Note that $(g_x)$ are coupled with the same $K$ in the construction.
\end{proof}
\begin{Remark}
The connection with $T_{\alpha,\frac{1}{2}}$ and a simple random walk on $\mathbb{Z}$ in Remark \ref{Zrandomwalk:00} shows another construction of $Y_{\frac{1}{2}}$. Let $(\xi_j)$ be \iid random variables in $[-1,1]$ with probability measure $\nu^\spin$. Let the random variable ${\mathfrak T}$ be equal to $k$ when the last return to the origin is at $2k < \infty$. Then $Y_{\frac{1}{2}}$ is the product of points in a point process (with $\alpha = 4pq$) with \pgfl
\[
G_{\frac{1}{2}}[f] = \mathbb{E}\big [\prod_{j=0}^{\mathfrak T}f(\xi_j)\big ]
= \frac{1}{\sqrt{\Big (1+\frac{4pq}{1-4pq}\int_{[-1,1]}\big (1-f(\xi)\big)\nu^\spin(d\xi)\Big )}}.
\]
The construction requires that $(S_n)$ be run forever to determine the last return to zero.
\end{Remark}
\subsection{Level set sums}\label{subsection:level}
Recall that level set sums from the origin are defined in Corollary \ref{symmetrycorr:000} as 
\[
\vartheta_v^{(N)}
= \sum_{y:\|y\|=v}g_y^{(N)},\ v\in \mathbb{R}_+.
\]
In Proposition \ref{proplevel:00} assume that there is a non-negative measure $\nu_*^\spin$ such that
\begin{equation}
\vartheta^{(N)}_{v} = {N\choose v}
 \frac{1}{2^{N/2}}\Big \{g_\emptyset +
\sum_{k=1}^{N}
\mathbb{E}\big [Y^k_{\frac{1}{2}}\big ]
 \sqrt{{N\choose k}}Q_k\big (v;N,\frac{1}{2}\big )\cdot \zeta_k\Big \},\ v \in \mathbb{N},
\label{symmetry:001a}
\end{equation}
where $(\zeta_k)$ are \iid N$(0,1)$. $Y$ and $Y_{\frac{1}{2}}$ are distributed respectively as the product of points in  \pgfl s
\begin{equation}
G\big [f\big ]=\frac {1}{1 + \int_\spin (1-f(\xi))\nu_*^\spin(d\xi)}
\label{Gstar:00}
\end{equation}
and $G_{\frac{1}{2}}[f]=G[f]^{\frac{1}{2}}$. The covariance function for $(\vartheta^{(N)}_{v})$ is for
$u,v \in \mathbb{N}$
\begin{equation}
\text{\rm Cov}\big (\vartheta^{(N)}_{u},\vartheta^{(N)}_{v}\big )
= {N\choose u}{N\choose v}
 \frac{1}{2^N}\Big \{1+
\sum_{k=1}^{N}
\mathbb{E}\big [Y^k\big ]
{N\choose k}Q_k\big (u;N,\frac{1}{2}\big )Q_k\big (v;N,\frac{1}{2}\big )\Big \},
 \label{covstarlevel:00}
\end{equation}
noting the unusual equality $\mathbb{E}\big [Y^k\big ] = \mathbb{E}\big [Y_{\frac{1}{2}}^k\big ]^2$. 
A limit form for level sets will be found. The limit uses a central limit theorem for the index in the in the level sets. Hermite-Chebycheff polynomials $(H_k)_{k=0}^\infty$, orthogonal on the standard normal, are used in a Gaussian process representation.  A generating function for them is
\[
\sum_{k=0}^\infty \frac{1}{k!}w^kH_k(t) = e^{wt - \frac{1}{2}t^2}.
\]
\begin{Proposition}\label{proplevel:00}
Define a Gaussian process 
$(\varkappa_t)_{t\in \mathbb{R}}$ by
\begin{align}
\varkappa_t
&= \frac{1}{\sqrt{2\pi}}e^{-\frac{1}{2}t^2}
\sum_{k=0}^\infty \mathbb{E}\big [Y_{\frac{1}{2}}^k\big ]
\frac{1}{\sqrt{k!}}H_k(t)\cdot \zeta_k,
\label{decomposition:22}
\end{align}
where $(\zeta_k)$ are \iid N$(0,1)$ random variables.
The covariance function of $\big (\varkappa_t\big )$ is
\begin{align}
\text{\rm Cov}\big (\varkappa_t, \varkappa_s\big )
&= \frac{1}{2\pi}e^{-\frac{1}{2}(t^2+s^2)}
 \Big \{ 1 + \sum_{k=1}^\infty
\mathbb{E}\big [Y^k\big ]\frac{1}{k!}H_k(t)H_k(s)\Big \}\nonumber\\
&=\mathbb{E}\big [\mathfrak{n}(t,s;Y)\big ],
\label{normalmix:000}
\end{align}
where $\mathfrak{n}(t,s;\rho)$ is the standard bivariate normal density with correlation coefficient $\rho$. $Y$ is the product of points in a point process with \pgfl  (\ref{Gstar:00}).
The scaled Gaussian field  $
\frac{1}{2}\sqrt{\frac{N}{2^N}}\Big (\vartheta^{(N)}_{\lfloor \frac{N}{2}+ \frac{\sqrt{N}}{2}t\rfloor }\Big )_{\lvert t \rvert < \sqrt{N}}$ converges weakly to $(\varkappa_t)$.
Weak convergence is in the sense of finite-dimensional distributions converging.
\end{Proposition}
\begin{proof}
First note that $(\varkappa_t)$ is well defined because the covariance function (\ref{normalmix:000}) is well defined and finite for all $s,t\in \mathbb{R}$.
Consider the expansion (\ref{symmetry:001a}) when $v = \lfloor \frac{N}{2}+ \frac{\sqrt{N}}{2}t\rfloor$.
Multiplying by $\frac{1}{2}\sqrt{\frac{N}{2^N}}$ the common multiplier factor in the expansion is then
\[
{N \choose  \lfloor \frac{N}{2}+ \frac{\sqrt{N}}{2}t\rfloor} \frac{1}{2^N}\frac{\sqrt{N}}{2}
\to \frac{1}{\sqrt{2\pi}}e^{-\frac{1}{2}t^2}
\]
and
\[
\sqrt{{N\choose k}}Q_k\Big (\big \lfloor \frac{N}{2}+ \frac{\sqrt{N}}{2}t\big \rfloor ;N,\frac{1}{2}\Big ) \to \frac{1}{\sqrt{k!}}H_k(t).
\]
See, for example, \cite{DG2012}.
These two pointwise limits prove the proposition since the covariance function (\ref{covstarlevel:00}) converges to (\ref{normalmix:000}) and this is enough for weak convergence of Gaussian processes.
\end{proof}
\begin{Remark}
A canonical form for a continuous time Markov process $\big (X(\tau)\big )_{\tau \geq 0}$
that has transition function densities of $X(\tau)=s$, with $X(0)=t$ is
\begin{equation*}
g_\tau(s,t;\nu_*)=\frac{e^{-s^2/2}}{\sqrt{2\pi}}
\Big \{1 + \sum_{k=1}^\infty e^{-\lambda_k\tau}
\frac{1}{k!}H_k(t)H_k(s)\Big \},
\label{canonical:32}
\end{equation*}
where 
\[
\lambda_k = \int_{[-1,1]}\big ( 1 - \xi^k \big )\nu_*(d\xi),
\]
with $\nu_*$ a non negative measure on $\spin$ such that the integral is convergent at $1$. This characterization is due to \cite{B1954}. Such processes appear from limits in generalized Ehrenfest urn models \citep{DG2012}. If $\nu_*=\delta_{\{0\}}$, $(X(\tau))$ is the Ornstein-Uhlenbeck process. 
Now
\[
\int_{\mathbb{R}_+} e^{-\tau } e^{-\lambda_k \tau}d\tau
= \frac{1}{1+\int_{[-1,1]}\big (1 - \xi^k\big )\nu_*^\spin(d\xi)},
\]
therefore (\ref{normalmix:000}) in Proposition \ref{proplevel:00} is equal to
\[
\text{\rm Cov}\big (\varkappa_t, \varkappa_s\big )=\frac{e^{-t^2/2}}{\sqrt{2\pi}}
\int_{\mathbb{R}_+} e^{-\tau} g_\tau(s,t;\nu_*)d\tau.
\]
\end{Remark}

\begin{Remark}
A decomposition is $\varkappa_t = \varkappa^\odd_t + \varkappa^\even_t$,
where $\varkappa^\odd_t= \frac{1}{2}\big (\varkappa_t - \varkappa_{-t}\big )$,
$\varkappa^\even_t= \frac{1}{2}\big (\varkappa_t + \varkappa_{-t}\big )$ are independent and equal to the odd and even terms in (\ref{decomposition:22}).  The further decomposition into 
odd and even processes follows because $H_{2k+1}(t)$ and $H_{2k}(t)$ are respectively odd and even functions of $t$. 
\end{Remark}
\section{A complex Gaussian field transform}
Define a complex Gaussian field by
\begin{equation*}
\widehat{\varkappa}_\theta = \int_{\mathbb{R}} e^{i\theta t}\varkappa_t dt, \ \theta \in \mathbb{R},
\label{complex:001}
\end{equation*}
with notation as in Proposition \ref{proplevel:00}. It is natural to consider the transform because it leads to a relatively simple covariance function.
\begin{Proposition}\label{prop:complex}
$(\widehat{\varkappa}_\theta)$ has a covariance matrix
\begin{equation*}
\text{\rm Cov}\big (\widehat{\varkappa}_\theta ,\widehat{\varkappa}_\phi\big )
= \mathbb{E}\big [\widehat{\varkappa}_\theta \overline{\widehat{\varkappa}}_\phi\big ]
= e^{-\frac{1}{2}(\theta^2+\phi^2)}\mathbb{E}\big [e^{\theta\phi Y}\big ], 
\label{complexcv:00}
\end{equation*}
where $Y$ is the product of points in the \pgfl (\ref{Gstar:00}).
Then
$\widehat{\varkappa}_\theta$ is distributed as $U_\theta+iV_\theta$ where $U_\theta$ and $V_\theta$ are independent real Gaussian processes with 
$U_\theta = \int_{\mathbb{R}_+} e^{i\theta t}\varkappa_t^\even dt$,
$V_\theta = i\int_{\mathbb{R}_+} e^{i\theta t}\varkappa_t^\odd dt$
and
\begin{align}
 \text{\rm Cov}\big (U_\theta,U_\phi\big )
 &= 
 e^{-\frac{1}{2}\big (\theta^2+\phi^2\big )}
 \mathbb{E}\big [\cosh\big (\theta\phi Y\big )\big ]\nonumber \\
  \text{\rm Cov}\big (V_\theta,V_\phi\big )
 &= 
 e^{-\frac{1}{2}\big (\theta^2+\phi^2\big )}
\mathbb{E}\big [ \sinh\big (\theta\phi Y\big )\big ].
\label{covoddeven:00}
\end{align}
Plancherel's theorem shows that
\begin{align}
\mathbb{E}\int_{\mathbb{R}} \lvert\varkappa_s\rvert^2ds &= 
\mathbb{E}\int_{\mathbb{R}} \lvert\widehat{\varkappa}_\theta\rvert^2d\theta\nonumber \\
& = \mathbb{E}\Bigl [\sqrt { \frac{\pi}{1-Y} }\Bigr ],
\label{Parseval:00}
\end{align}
assuming that the expectation on the right of (\ref{Parseval:00}) is finite.
\end{Proposition}
\begin{proof}
The proof follows immediately from Proposition \ref{proplevel:00}, (\ref{normalmix:000}) and the form of a bivariate normal characteristic function.
Taking care with the complex conjugate in the covariance function 
\begin{align*}
\text{Cov}\big (\hat{\varkappa}_\theta,\hat{\varkappa}_\theta\big )
&= \mathbb{E}\big [\hat{\varkappa}_\theta\overline{\hat{\varkappa}}_\phi\big ]\\
&= \int_{\mathbb{R}}\int_{\mathbb{R}}
e^{i\theta t - i\phi s}\mathbb{E}\big [f(t,s;Y)\big ]dtds\\
&= e^{-\frac{1}{2}(\theta^2+\phi^2)}\mathbb{E}\big [e^{-i^2\theta\phi Y}\big ]\\
&= e^{-\frac{1}{2}(\theta^2+\phi^2)}\mathbb{E}\big [e^{\theta\phi Y}]
\end{align*}
For (\ref{covoddeven:00})
\begin{align*}
\mathbb{E}\big [
U_\theta U_\phi\big ]
&= \frac{1}{4}e^{-\frac{1}{2}\big (\theta^2+\phi^2\big )}\mathbb{E}\big [e^{-\theta\phi Y} + e^{\theta\phi Y} + e^{\theta\phi Y} + e^{-\theta\phi Y}\big ]\\
&=  e^{-\frac{1}{2}\big (\theta^2+\phi^2\big )}
 \mathbb{E}\big [\cosh\big (\theta\phi Y\big )\big ]
 \end{align*}
 and similarly
\begin{align*}
\mathbb{E}\big [
V_\theta V_\phi \big ] &= 
-\frac{1}{4}e^{-\frac{1}{2}\big (\theta^2+\phi^2\big )}
\mathbb{E}\big [e^{-\theta\phi Y} - e^{\theta\phi Y} - e^{\theta\phi Y} + e^{-\theta\phi Y}\big ]\\
&=e^{-\frac{1}{2}\big (\theta^2+\phi^2\big )}
\mathbb{E}\big [ \sinh\big (\theta\phi Y\big )\big ].
\end{align*}
The identity (\ref{Parseval:00}) follows from 
\begin{align*}
\mathbb{E}\int_{\mathbb{R}} \lvert\widehat{\varkappa}_\theta\rvert^2d\theta
&= \mathbb{\int_{\mathbb{R}} \text{Var}(\widehat{\varkappa}_\theta})d\theta\nonumber\\ 
&=\int_{\mathbb{R}} e^{-\theta^2}\mathbb{E}\big [e^{\theta^2Y}\big ]d\theta\\
&= \mathbb{E}\Bigl [\sqrt { \frac{\pi}{1-Y} }\Bigr ].
\end{align*}
\end{proof}
\begin{Corollary}
A strong representation of $U_\theta, V_\theta$ in Proposition \ref{prop:complex} is
\begin{align}
U_\theta &= \sum_{k=0}^\infty
\frac{ \text{\rm Poisson}\big (k;\frac{1}{2}\theta^2\big ) 
 }
{ \sqrt{(2k-1+\delta_{k0})!!}  }
\mathbb{E}\big [Y_{\frac{1}{2}}^{2k}\big ](-\sqrt{2})^k  \zeta_{2k}
\nonumber\\
V_\theta &= \theta\sum_{k=0}^\infty
\frac{ \text{\rm Poisson}\big (k;\frac{1}{2}\theta^2\big )   }
{ \sqrt{(2k+1)!! } }
\cdot \mathbb{E}\big [Y_{\frac{1}{2}}^{2k+1}\big ](-\sqrt{2})^k \zeta_{2k+1},
\label{complex:88}
\end{align}
where 
$
\text{\rm Poisson}\big (k;\lambda\big ) = \frac{\lambda^k}{k!}e^{-\lambda},\ k \in \mathbb{N}.
$ \\[0.2cm]
An inversion formula is
\begin{equation}
\varkappa_t = \frac{1}{2\pi}\int_{\mathbb{R}}
e^{-i\theta t}\widehat{\varkappa}_\theta\ d\theta,\ t \in \mathbb{R}.
\label{inversion:23}
\end{equation}
\end{Corollary}
\begin{proof}
The representations in (\ref{complex:88}) follow from taking the transform of (\ref{decomposition:22})  and identifying real and imaginary parts. When $X$ is N$(0,1)$,
$
\mathbb{E}\big [e^{i\theta X}H_k(X)\big ]
 = e^{-\frac{1}{2}\theta^2}(i\theta)^k, k \in \mathbb{N}.
$
Taking the transform
\begin{align}
\widehat{\varkappa}_\theta & =\int_{\mathbb{R}}e^{i\theta t}\varkappa_tdt\nonumber\\
& = e^{-\frac{1}{2}\theta^2}
\sum_{k=0}^\infty 
\mathbb{E}\big [Y_{\frac{1}{2}}^k\big ]\frac{1}{\sqrt{k!}}(i\theta)^k\zeta_k\label{Poissonpf:01}\\
&= e^{-\frac{1}{2}\theta^2}\sum_{k=0}^\infty 
\mathbb{E}\big [Y_{\frac{1}{2}}^{2k}\big ]\frac{1}{\sqrt{(2k)!}}(-1)^k\theta^{2k}\zeta_{2k}\nonumber\\
&+ie^{-\frac{1}{2}\theta^2}\sum_{k=0}^\infty 
\mathbb{E}\big [Y_{\frac{1}{2}}^{2k+1}\big ]\frac{1}{\sqrt{(2k+1)!}}(-1)^k\theta^{2k+1}\zeta_{2k+1}.
\nonumber
\end{align}
The Poisson expressions now follow using the identities
\begin{align*}
\frac{1}{\sqrt{(2k)!} }= \frac{ \sqrt{2}^k }{2^k}\cdot 
\frac{1}{ \sqrt{(2k-1+\delta_{k0})!!}  }\text{~and~}
\frac{1}{\sqrt{(2k+1)!}} = \frac{\sqrt{2}^k}{2^k}\cdot \frac{1}{\sqrt{(2k+1)!!}}\ .
\end{align*}
To find the inversion formula (\ref{inversion:23}) invert the series (\ref{Poissonpf:01}).
Note that 
\[
\frac{1}{\sqrt{2\pi}}\int_{\mathbb{R}} e^{-i\theta t} e^{-\frac{1}{2}\theta^2}(i\theta)^k d\theta
\]
is the coefficient of $\frac{z^k}{k!}$ in 
\[
\frac{1}{\sqrt{2\pi}}\int_{\mathbb{R}} e^{-i\theta t}
 e^{-\frac{1}{2}\theta^2}e^{i\theta z}d\theta
 = e^{-\frac{1}{2}t^2}\cdot e^{-\frac{1}{2}z^2+zt},
\]
which is $e^{-\frac{1}{2}t^2}H_k(t).$ 
Therefore from (\ref{Poissonpf:01}) and (\ref{decomposition:22})
\[
\frac{1}{2\pi}\int_{\mathbb{R}} e^{-i\theta t}
\widehat{\varkappa}_\theta\ d\theta
= \frac{1}{\sqrt{2\pi}}e^{-\frac{1}{2}t^2}\sum_{k=0}^\infty 
\mathbb{E}\big [Y_{\frac{1}{2}}^k\big ]\frac{1}{\sqrt{k!}}H_k(t)\cdot\zeta_k = \varkappa_t.
\]
\begin{Remark}
Let $W_\theta^u$ and $W_\theta^v$ be independent Poisson$(\frac{1}{2}\theta^2)$ random variables. Then marginally with respect to $U_\theta, V_\theta$
\begin{align}
U_\theta &= 
\mathbb{E}_{W_\theta^u}\Big [
\Big ((2W_\theta^u-1+ I\big \{W_\theta^u=0\big \})!!\Big )^{-\frac{1}{2}}
\mathbb{E}\big [Y_{\frac{1}{2}}^{2W_\theta^u}\big ](-\sqrt{2})^{W_\theta^u } \zeta_{2W_\theta^u}\Big ]
\nonumber\\
V_\theta &= \theta
\mathbb{E}_{W_\theta^v}\Big [
\Big ((2W_\theta^v+1)!! \Big )^{-\frac{1}{2}}
\cdot \mathbb{E}\big [Y_{\frac{1}{2}}^{2W_\theta^v+1}\big ](-\sqrt{2})^{W_\theta^v} \zeta_{2W_\theta^v+1}\Big ].
\nonumber
\end{align}
These random sums can be extended to a collection $\theta_1,\theta_2,\ldots, \theta_r$ with independent $\big (W_{\theta_j}^u, W_{\theta_j}^v\big )_{j\in [r]}$.
\end{Remark}

\end{proof}
\section{Acknowledgements}
We thank Andrea Collevecchio  and Shuhei Mano for their very helpful extensive comments. 
Two referees to thank gave a careful reading of the paper, corrections and connections.
A big thanks to Persi Diaconis for his suggestion to consider Gaussian Fields on a hypercube constructed from long range random walks. The research in this paper was stimulated by \cite{D2020}.

\section{Funding}
This research did not receive any specific grant from funding agencies in the public, commercial, or not-for-profit sectors.


\begin{thebibliography}{99}
\bibitem[Abramowitz and Stegun(1972)]{AS1972}
Abramowitz, M., Stegun, I. A. 1972. Handbook of mathematical functions with formulas, graphs, and mathematical tables. U S Department of Commerce.
%
\bibitem[Bochner(1954)]{B1954}
Bochner, S. 1954.
Positive zonal functions on spheres, Proc. Nat. Acad. Sci. USA. 40,  1141--1147.
%
\bibitem[Berestycki and Powell(2024)]{B2016}
 Berestycki, N., and Powell, E. 2024. Gaussian free field and Liouville quantum gravity. arXiv preprint arXiv:2404.16642.
%
\bibitem[Collevecchio and  Griffiths(2021)]{CG2021}
Collevecchio, A. and Griffiths, R., 2021. A class of random walks on the hypercube.  In and Out of Equilibrium 3: Celebrating Vladas Sidoravicius. Eulália Vares, M., Fernández, R., Renato Fontes, L. and Newman, C. M. (eds.). Cham Switzerland: Springer, Vol. 77. p. 265--298 34 p. (Progress in Probability; vol. 77). 
%
\bibitem[Collevecchio and  Griffiths(2025)]{CG2025}
Collevecchio, A. and Griffiths, R., 2025. Gaussian free fields on the hypercube.
J. Theor. Probab. 2025--06, 38, Article 35.
%
\bibitem[Daley and Vere-Jones(2008)]{DVJ2003}
Daley, D. J., and Vere-Jones, D. 2003. An introduction to the theory of point processes. Vol. I, Elementary theory and methods (2nd ed.). Springer.
%
\bibitem[Daley and Vere-Jones(2008)]{DVJ2008} Daley, D. J.,  Vere-Jones, D. 2008. An introduction to the theory of point processes. Vol. II, General theory and structure (2nd ed.). Springer.
%
\bibitem[Diaconis(2020)]{D2020} Diaconis, P. 2020.  Gaussian Fields on the hypercube.  Unpublished note. 
%
\bibitem[Diaconis and Evans(2002)]{DE2002}
 Diaconis P. and Evans, S. N. 2002. A different construction of Gaussian fields from Markov Chains: Dirichlet covariances.
Ann. I. H. Poincar{\'e}, 38 863--878.
\bibitem[Diaconis and Griffiths(2012)]{DG2012}
Diaconis, P. and Griffiths R. C. 2012. Exchangeable pairs of Bernoulli random variables, Krawtchouk polynomials, and Ehrenfest urns. Aust. NZ J .Stat. 54, 81--101.
%
\bibitem[Dynkin(1984)]{D1984}
Dynkin, E. B. 1984. Gaussian and non-Gaussian random fields associated with Markov processes. Journal of functional analysis, 55, 344--376. 

\bibitem[Feller(1991)]{F1991}
Feller, W. 1991. An Introduction to Probability Theory and its Applications. Volume 1, 3rd edition.
John Wiley and Sons. Inc., New York.

\bibitem[Gregoire(1984)]{G1984}
Gregoire, G. 1984. Negative binomial distributions for point processes. Stochastic Processes and Their Applications, 16, 179--188. 

\bibitem[Marcus and Rosen(2006)]{MR2006}
Marcus, M. B., Rosen, J. 2006. Markov processes, Gaussian processes, and local times. Cambridge University Press.

\bibitem[Talagrand(2014)]{T2014}
Talagrand, M. 2014. Mean field models for spin glasses (2nd ed., Vol. 54). Springer.
%
  \bibitem[Werner and Powell(2020)]{WP2020}
 Werner, W and Powell, E. Lecture notes on the Gaussian free field.\\
 \url{https://arxiv.org/abs/2004.04720}
 %
 \bibitem[Westcott(1972)]{W1972}
 Westcott, M. 1972. The probability generating functional. J. Aust. Math. Soc., 14, 448-466.
\end{thebibliography}
\end{document}